\documentclass{article}

\usepackage{authblk}

\usepackage[a4paper, total={7in, 8.5in}]{geometry}

\usepackage{amssymb,amsmath,amsfonts,amsthm}
\usepackage{color}
\usepackage[dvipsnames]{xcolor}
\usepackage{graphicx}
\usepackage{placeins}
\usepackage[dvipsnames]{xcolor}

\newtheorem{theorem}{Theorem}[section]
\newtheorem*{theorem*}{Theorem}
\newtheorem{lemma}[theorem]{Lemma}
\newtheorem{proposition}[theorem]{Proposition}

\newtheorem{remark}[theorem]{Remark}

\newtheorem{definition}[theorem]{Definition}

\numberwithin{equation}{section}

\newcommand{\normsymb}{\|}
\newcommand{\norm}[2]{\normsymb{#1}\normsymb_{#2}}  % norms
\renewcommand{\epsilon}{\varepsilon}
\newcommand{\RR}{\mathbb{R}}
\newcommand*{\CC}{\mathbb{C}}

\newcommand*{\NN}{\mathbb{N}}

\usepackage{mathtools}
\usepackage{hyperref}
\hypersetup{hidelinks}
\usepackage{soul}

\newcommand{\dvm}{\,{\rm dvol}_M}
\newcommand{\dvpm}{\,{\rm dvol}_{\partial M}}

\newcommand{\pam}{{\partial M}}
\allowdisplaybreaks

\title{Magnetic Steklov operator on differential forms}

\author[1]{Tirumala Chakradhar\thanks{\texttt{tirumala.chakradhar@bristol.ac.uk}}}
\author[2]{Katie Gittins\thanks{\texttt{katie.gittins@durham.ac.uk}}}
\author[3,4]{Georges Habib\thanks{\texttt{ghabib@ul.edu.lb}}}
\author[2]{Norbert Peyerimhoff\thanks{\texttt{norbert.peyerimhoff@durham.ac.uk}}}

\affil[1]{\footnotesize School of Mathematics, University of Bristol, Fry Building, Woodland Road, BS8 1UG, UK}
\affil[2]{\footnotesize Department of Mathematical Sciences, Durham University, Mathematical Sciences and Computer Science Building, Upper Mountjoy Campus, Stockton Road, Durham University, DH1 3LE, United Kingdom}
\affil[3]{\footnotesize Lebanese University, Faculty of Sciences II, Department of Mathematics, P.O. Box 90656 Fanar-Matn, Lebanon}
\affil[4]{\footnotesize Universit\'e de Lorraine, CNRS, IECL, 54506 Nancy, France}

\date{\today}

\begin{document}
%\parindent

\maketitle

\begin{abstract}
     In this paper, we introduce the magnetic Steklov operator on differential forms and show that the underlying boundary value problem is well-posed. Moreover, we show that an analogue of the Diamagnetic Inequality does not always hold for this operator, and we present some spectral computations of magnetic Steklov operators for $2$-dimensional and $4$-dimensional balls in Euclidean space.
\end{abstract}

\tableofcontents

\section{Introduction}

Let $(M^m,g)$ be a compact Riemannian manifold of dimension $m$ with smooth boundary $\partial M$. 
Fix a smooth, real $1$-form $\eta$ (called the magnetic potential) on $M$. 
Let $k\in \{0,\ldots,m\}$.
We denote by $\Omega^k(M,\mathbb{C})$ the space of complex differential $k$-forms on $M$, and we denote by $$\Delta^\eta=d^\eta\delta^\eta+\delta^\eta d^\eta$$ the magnetic Hodge Laplacian on such $k$-forms.
Here $d^\eta:=d+i\eta\wedge$ is the magnetic exterior differential on $M$ and $\delta^\eta=\delta-i\eta\lrcorner$ is the formal $L^2$ adjoint of $d^\eta$. A complex form $\omega$ is called {\it magnetic harmonic} if $\Delta^\eta\omega=0$. 

The Steklov eigenvalue problem has received a great deal of attention in recent years, both for functions and for differential forms. See, for example, \cite{CGGS:24} for a recent survey. 
Moreover, the magnetic Steklov eigenvalue problem for functions is a dynamic area of current research, see, for example, 
\cite{CS:24}, \cite{CPS:22}, 
\cite{DKSU:07}, \cite{EO:22}, \cite{H:18},
\cite{HN:24a}, \cite{HN:24b}, \cite{LT:23}, \cite{NSU:95}, \cite{PS:23}.

Motivated by our previous work on the spectrum of the magnetic Hodge Laplacian on differential forms \cite{EGHP:23}, and the spectrum of the magnetic Steklov operator on functions \cite{CGHP:25}, the goal of this paper is to introduce the magnetic Steklov operator on differential forms. This operator is the magnetic analogue of the Steklov operator on differential forms that was studied by Raulot and Savo in \cite{RS:12}.

The first aim of this paper is to show that the spectral problem associated with the magnetic Steklov operator on differential forms is well posed. To that end, in Section \ref{sec:wellposednesssteklovforms}, we prove the following result:

\begin{theorem}\label{thm:magneticboundprob}
Let $(M^m,g)$ be a compact Riemannian manifold of dimension $m$ with smooth boundary $\partial M$ and let $\eta$ be a smooth, real differential $1$-form on $M$. The boundary value problem 
\begin{equation}\label{eq:magneticsteklovsolut}
\left\{
\begin{matrix}
	\Delta^\eta\omega=\varphi & \text{on $M$,}\\\\
	\omega=\psi & \text{on $\partial M$,}
\end{matrix}\right.
\end{equation}    
has a unique smooth solution for each $\varphi\in \Omega^k(M,\mathbb{C})$ and $\psi\in \Omega^k(M,\mathbb{C})|_{\partial M}$.
\end{theorem}

Using this result, it follows that for any differential form $\psi\in \Omega^k(\partial M,\mathbb{C})$, there is a unique $\hat\omega \in \Omega^k(M,\mathbb{C})$ such that 
\begin{equation}\label{eq:steklovmagnetic}
\left\{
\begin{matrix}
	\Delta^\eta\hat\omega=0 & \text{on $M$,}\\\\
	\iota^*\hat\omega=\psi,\,\, \nu\lrcorner\hat\omega=0 & \text{on $\partial M$.}
\end{matrix}\right.
\end{equation} 
We call $\hat\omega$ the $\eta$-harmonic extension of $\psi$.

The magnetic Steklov operator $T^{[k]}:\Omega^k(\partial M,\mathbb{C})\to \Omega^k(\partial M,\mathbb{C})$ on $k$-forms is then defined as:
$$T^{[k],\eta}\omega:=-\nu\lrcorner d^\eta\hat\omega.$$

In Section \ref{ss:diamagnetic}, we consider the spectral properties of the magnetic Steklov operator on differential forms.
We prove that it has an increasing sequence of eigenvalues 
        $$\sigma_{1,k}^\eta(M)\leq \sigma_{2,k}^\eta(M)\leq \ldots,$$
with finite multiplicities accumulating at $+\infty$.
Moreover, we show that an analogue of the Diamagnetic Inequality does not always hold for this operator. More precisely, we prove the following theorem.

\begin{theorem} \label{thm:eigtaylor}
  Let $(M^m,g)$ be a compact, oriented Riemannian manifold of dimension $m$ with smooth boundary and let $\eta$ be a magnetic potential. 
  Then, for any $t \in \RR$, we have,
  \begin{equation} \label{eq:eigvalcomp}
  \sigma^{t\eta}_{1,k}(M)\leq \sigma_{1,k}(M)+\frac{2t}{||\hat\omega||^2_{L^2(\partial M)}}{\rm Im}\left(\int_M\langle\mathcal{L}_\eta\hat\omega,\hat\omega\rangle {\rm dvol}_M\right)+\frac{||\hat\omega||^2_{L^2}||\eta||_\infty^2}{||\hat\omega||^2_{L^2(\partial M)}}t^2
  \end{equation}
  where $\hat\omega \in \Omega^k(M,\CC)$ is the harmonic extension of an eigenform $\omega$ of the Steklov operator $T^{[k]}$ (linearly extended to complex $k$-forms) associated with the eigenvalue $\sigma_{1,k}(M)$. In particular, if ${\rm Im} \left( \int_M \langle \mathcal{L}_{\eta} \hat\omega, \hat\omega \rangle {\rm dvol}_M\right)$ is negative for some complex eigenform $\hat\omega$, then, for small positive $t$, we get that
  $$ \sigma^{t\eta}_{1,k}(M)<\sigma_{1,k}(M).$$
\end{theorem}

Furthermore, in the next theorem we provide an example where the condition in Theorem \ref{thm:eigtaylor} holds. This shows that the Diamagnetic Inequality is not always satisfied. 
\begin{theorem} \label{thm:balldiamineqcounter}
For $n\geq 1$, let $M=\mathbb{B}^{2n}$ be the unit ball in $\mathbb{R}^{2n}$ equipped with the metric $dr^2\oplus r^2g_{\mathbb{S}^{2n-1}}$ with $r=|x|$ where $x=(x_1,y_1,x_2, y_2,\ldots, x_n,y_n)\in \mathbb{R}^{2n}$. Let $\eta$ be the magnetic potential on the ball given by $\eta=\sum_{j=1}^n(-y_j\partial_{x_j}+x_j\partial_{y_j})$.  Then, for small $t>0$, we have 
$$ \sigma^{t\eta}_{1,1}(M)<\sigma_{1,1}(M).$$
\end{theorem}

Finally, in Section \ref{sec:examples}, we present some explicit spectral computations of magnetic Hodge Laplacians and magnetic Steklov operators on differential forms on some Euclidean spheres and balls.

\section{Magnetic Steklov operator on differential forms} \label{sec:wellposednesssteklovforms}
In this section, we prove existence and uniqueness for the magnetic Steklov problem \eqref{eq:steklovmagnetic} for differential forms.

\subsection{Uniqueness of a solution}\label{ss:unique}
In this section, we prove the following result which gives uniqueness in \eqref{eq:steklovmagnetic}. 

\begin{theorem} \label{thm:dirichletmagnetic} Let $(M^m,g)$ be a compact Riemannian manifold of dimension $m$ with smooth boundary $\partial M$. If $\omega$ is a complex differential $k$-form which is magnetic harmonic on $M$ and vanishes on $\partial M$, then $\omega$ vanishes everywhere. 
\end{theorem}

The proof of this theorem for the case without magnetic potential was carried out by C. Ann{\'e} in \cite{Ann:89}. To prove this theorem in the magnetic case, we follow the same strategy and make the necessary modifications to account for the terms involving the magnetic potential.

\begin{proof}
The proof of this theorem is divided into several steps. In the first step, we prove that $\omega$ vanishes to infinite order at $t=0$ in a tubular neighborhood of the boundary. In the second step, we use a technical lemma (which is analogous to the non-magnetic version given in \cite{Ann:89}) to show that $\omega$ vanishes on a small tubular neighborhood of the boundary. In the last step, we construct a double copy of the manifold to show that $\omega$ vanishes on some open set of a closed manifold and deduce that it therefore vanishes everywhere. \\

\noindent {\bf Step 1}:  
We consider a tubular neighborhood of the boundary, that is a manifold of the form $[0,t_0]\times \partial M$ equipped with the metric $dt^2\oplus g_t$, where $t$ is the distance to the boundary in normal coordinates. Henceforth, we canonically identify the tubular neighborhood with $[0,t_0]\times \partial M$. The form $\omega$ can be written as $\omega=\omega_1(t,\cdot)+dt\wedge\omega_2(t,\cdot)$ where $\omega_1,\omega_2$ are differential forms on $\partial M\times \{t\}$ and the magnetic potential $\eta$ can be written as $\eta=\eta_1+\eta_2 dt$. Since $\omega$ vanishes on $\partial M$, the forms $\omega_1$ and $\omega_2$ both vanish at $t=0$. In the following, we denote by $'$ the derivative with respect to $t$ and by $d^N$ the exterior differential on $\partial M$. 
A straightforward computation shows that for any differential $k$-form $\gamma\in \Omega^k(\partial M\times \{t\}, \mathbb{C})$, the following identity $d\gamma=d^{N}\gamma+dt\wedge \gamma'$ holds, where we have set $\gamma'(X_1,\ldots,X_k):=\frac{d}{dt}(\gamma(X_1,\ldots,X_k))$ for smooth vector fields $X_1,\ldots,X_k$ on $\partial M$.  This gives
\begin{equation}\label{eq:exteriordiffmagnetic}
d^\eta\omega=d^{\eta_1}\omega_1+dt\wedge (\omega_1'-d^{\eta_1}\omega_2+i\eta_2\omega_1)
\end{equation}
where $d^{\eta_1}:=d^N+i\eta_1\wedge$. 

Now, we compute the codifferential $\delta^\eta$ (i.e. the formal adjoint of $d^\eta$). We denote by $\delta_t^{\eta_1}:=\delta_t-i\eta_1\lrcorner$ the formal adjoint of $d^{\eta_1}$ with respect to the metric $g_t$ where $\delta_t$ is the formal adjoint of $d^{\partial M}$ with respect to $g_t$.  For any $\widetilde\omega=\widetilde\omega_1+dt\wedge\widetilde\omega_2$ with support in $(0,t_0)\times \partial M$, we have, by \eqref{eq:exteriordiffmagnetic},
\begin{eqnarray}\label{eq:adjointdmagnetic}
    \int_M\langle d^\eta\omega,\widetilde\omega\rangle {\rm dvol}_M&=&\int_M\langle d^{\eta_1}\omega_1,\widetilde\omega_1\rangle {\rm dvol}_M+\int_M\langle\omega_1'-d^{\eta_1}\omega_2+i\eta_2\omega_1,\widetilde\omega_2\rangle {\rm dvol}_M\nonumber\\
    &=&\int_M\langle \omega_1,\delta_t^{\eta_1}\widetilde\omega_1\rangle {\rm dvol}_M+\int_M\langle \omega_1',\widetilde\omega_2\rangle {\rm dvol}_M\nonumber\\&&-\int_M\langle\omega_2,\delta_t^{\eta_1}\widetilde\omega_2\rangle {\rm dvol}_M-\int_M\langle\omega_1,i\eta_2\widetilde\omega_2 \rangle{\rm dvol}_M\nonumber\\
    &=&\int_M\langle \omega,\delta_t^{\eta_1}\widetilde\omega_1\rangle {\rm dvol}_M+\int_M\langle \omega_1',\widetilde\omega_2\rangle {\rm dvol}_M\nonumber\\&&-\int_M\langle\omega,dt\wedge\delta_t^{\eta_1}\widetilde\omega_2\rangle {\rm dvol}_M-\int_M\langle\omega,i\eta_2\widetilde\omega_2 \rangle{\rm dvol}_M\nonumber.\\
\end{eqnarray}
We now compute the second integral in the r.h.s. of the equality above. To simplify some computations, we introduce the positive symmetric endomorphism $A_t$ of the tangent space of $\partial M$ given by $g_t(X,Y)=g_0(A_tX,Y)$ where $g_0=g$ is the metric on the manifold $\partial M$ induced from the metric on $M$. Clearly the endomorphism $A_t$ on the tangent space extends to a symmetric endomorphism $B_t$ on complex differential $k$-forms on $\partial M$ in a way that $\langle\omega_1,\omega_2\rangle_{t}=\langle B_t\omega_1,\omega_2\rangle_{0}$. In the following, we denote by $v(t,y)$ the function given by ${\rm dvol}_M=v(t,y)dt\,{\rm dvol}_{\partial M}$. We have 
\begin{eqnarray*}
    0&=&\int_0^{t_0}\frac{d}{dt}\left(\langle \omega_1,\widetilde\omega_2\rangle_{t} v(t,y)\right) dt\\
    &=&\int_0^{t_0}\frac{d}{dt}\left(\langle B_t\omega_1,\widetilde\omega_2\rangle_{0} v(t,y)\right) dt\\
    &=&\int_0^{t_0}(\langle B_t'\omega_1,\widetilde\omega_2\rangle_{0} v(t,y)+\langle B_t\omega_1',\widetilde\omega_2\rangle_{0}v(t,y)+\langle B_t\omega_1,\widetilde\omega_2'\rangle_{0}v(t,y)%\\&&
    +\langle B_t\omega_1,\widetilde\omega_2\rangle_{0} v'(t,y)) dt\\
    &=&\int_0^{t_0}(\langle B_t^{-1}B_t'\omega_1,\widetilde\omega_2\rangle_{t} v(t,y)+\langle \omega_1',\widetilde\omega_2\rangle_{t}v(t,y)+\langle \omega_1,\widetilde\omega_2'\rangle_{t}v(t,y)%\\&&
    +\langle \omega_1,\widetilde\omega_2\rangle_{t} v'(t,y)) dt.
\end{eqnarray*}
After integrating over the boundary, we deduce that 
\begin{eqnarray}\label{eq:derivativeintegralmagnetic}
    \int_M\langle \omega_1',\widetilde\omega_2\rangle{\rm dvol}_M&=&-\int_M\langle B_t^{-1}B_t'\omega_1,\widetilde\omega_2\rangle{\rm dvol}_M-\int_M\langle \omega_1,\widetilde\omega_2'\rangle{\rm dvol}_M%\nonumber\\&&
    -\int_M\langle \omega_1,\widetilde\omega_2\rangle \frac{v'(t,y)}{v(t,y)} {\rm dvol}_M.
\end{eqnarray}
Plugging Equation \eqref{eq:derivativeintegralmagnetic} into \eqref{eq:adjointdmagnetic} yields the following formula for $\delta^{\eta}\widetilde\omega$ with $\widetilde\omega=\widetilde\omega_1+dt\wedge\widetilde\omega_2$:
\begin{equation}\label{eq:deltaalphamagnetic}
\delta^\eta\widetilde\omega=\delta_t^{\eta_1}\widetilde\omega_1-\left(B'_tB_t^{-1}+\frac{v'(t,y)}{v(t,y)}+i\eta_2\right)\widetilde\omega_2-\widetilde\omega_2'-dt\wedge\delta_t^{\eta_1}\widetilde\omega_2.
\end{equation}
Since $\omega$ is magnetic harmonic and vanishes on $\partial M$, by the magnetic Stokes formula it follows that
\begin{equation*}
    0 = \int_M \langle \Delta^\eta \omega, \omega\rangle {\rm dvol}_M 
    = \int_M \langle d^\eta \omega, d^\eta \omega\rangle {\rm dvol}_M
    + \int_M \langle \delta^\eta \omega, \delta^\eta \omega\rangle {\rm dvol}_M,
\end{equation*}
and hence $d^\eta\omega=0$  and $\delta^\eta\omega=0$. 
Using Equation \eqref{eq:exteriordiffmagnetic}, we get
\begin{equation}\label{eq:vanishing1}
    d^{\eta_1}\omega_1=0,\,\, \omega_1'=d^{\eta_1}\omega_2-i\eta_2\omega_1.
\end{equation} 
Equation \eqref{eq:deltaalphamagnetic} gives the two other equations 
\begin{equation}\label{eq:vanishing2}
    \delta_t^{\eta_1}\omega_2=0,\,\, \omega_2'=\delta_t^{\eta_1}\omega_1-\left(B'_tB_t^{-1}+\frac{v'(t,y)}{v(t,y)}+i\eta_2\right)\omega_2.
\end{equation}
The four equations in \eqref{eq:vanishing1} and \eqref{eq:vanishing2} allow us to express higher order derivatives in terms of lower order derivatives and hence give us that $\omega_1$ and $\omega_2$ vanish to infinite order at $t=0$.  

{\bf Step 2:} In this step, we prove that $\omega$ vanishes in a small tubular neighborhood of the boundary. For this, we need the following technical lemma (which is analogous to the non-magnetic version given in \cite{Ann:89}): 
\begin{lemma}\label{lem:magneticineq}
Let $(M^m,g)$ be a Riemannian manifold of dimension $m$ with smooth boundary and let $\eta$ be a real $1$-form on $M$. Let $T>0$ be such that there is a canonical isometry between $[0,T]\times \partial M$ and a tubular neighborhood in $M$. Then there exists a $t_0\in (0,T)$ such that for any complex differential $k$-form $\alpha$ that vanishes to infinite order at $t=0$ and with support in a tubular neighborhood $(0,t_0)\times \partial M$ of the boundary, we have 
$$\int_M t^{-n-1}|\alpha|^2{\rm dvol}_M\leq 4 \int_M t^{-n+1}(|d^\eta\alpha|^2+|\delta^\eta\alpha|^2){\rm dvol}_M,$$
for any $n\in \NN$, $n \geq 2$.
\end{lemma}
The proof of this lemma is very technical and relies on several inequalities.
We follow the same strategy of proof as that given for the non-magnetic case in \cite{Ann:89}.

\begin{proof} Let $t_0\in (0,T)$ that we will determine later, and $\alpha=\alpha_1+dt\wedge\alpha_2$ be a complex differential $k$-form as stated in the lemma. Before computing the norms of $d^\eta\alpha$ and $\delta^\eta\alpha$, we will prove the following inequality 
\begin{equation}\label{eq:inequalityalphaalpha'magnetic}
\left(\frac{n-M_1t_0}{2}\right)^2\int_M t^{-n-1}|\alpha|^2{\rm dvol}_M\leq \int_M t^{-n+1}|\alpha'|^2{\rm dvol}_M,
\end{equation}
for some constant $M_1$ that we will determine later and $n \in \mathbb{N}, n \geq 2$. Indeed, since $\alpha$ vanishes to infinite order at $t=0$, and the support of $\alpha$ is in $(0,t_0)\times \partial M$, we have
\begin{eqnarray*}
0&=&\int_0^{t_0}\frac{d}{dt}\left(t^{-n}|\alpha|^2 v(t,y)\right)dt\\
&=&-n\int_0^{t_0}t^{-n-1}|\alpha|^2 v(t,y)dt+\int_0^{t_0}t^{-n}\frac{d}{dt}(|\alpha_1|^2+|\alpha_2|^2)v(t,y)dt%\\&&
+\int_0^{t_0}t^{-n}|\alpha|^2 v(t,y)'dt.
\end{eqnarray*}
Recall here that $v$ is the function given by ${\rm dvol}_M=v(t,y)dt\,{\rm dvol}_{\partial M}$. Now, we have that 
$$\frac{d}{dt}(|\alpha_1|^2)=\frac{d}{dt}(\langle B_t\alpha_1,\alpha_1\rangle_{0})=\langle B_t^{-1}B_t'\alpha_1,\alpha_1\rangle_{t}+2\Re(\langle\alpha_1',\alpha_1\rangle_{t}).$$
The same computation can be done for $\alpha_2$. This allows us to get the following 
\begin{eqnarray*}
n\int_0^{t_0}t^{-n-1}|\alpha|^2 v(t,y)dt&=&\int_0^{t_0}t^{-n}(\langle B_t^{-1}B_t'\alpha_1,\alpha_1\rangle_{t}+\langle B_t^{-1}B_t'\alpha_2,\alpha_2\rangle_{t})v(t,y)dt\\&&+2\int_0^{t_0}t^{-n}\Re(\langle\alpha_1',\alpha_1\rangle_{t}+\langle\alpha_2',\alpha_2\rangle_{t})v(t,y)dt\\&&+\int_0^{t_0}t^{-n}|\alpha|^2 v(t,y)'dt.
\end{eqnarray*}
Integrating over $(\partial M,g)$, using the fact that $\alpha=0$ on $M\setminus((0,t_0)\times \partial M)$, and the Cauchy-Schwarz inequality yields 
\begin{eqnarray*}
n\int_Mt^{-n-1}|\alpha|^2 {\rm dvol}_M&=&\int_Mt^{-n}(\langle B_t^{-1}B_t'\alpha_1,\alpha_1\rangle_{t}+\langle B_t^{-1}B_t'\alpha_2,\alpha_2\rangle_{t}){\rm dvol}_M\\&&+2 \int_Mt^{-n}\Re(\langle\alpha',\alpha\rangle_{t}){\rm dvol}_M+\int_Mt^{-n}|\alpha|^2 \frac{v(t,y)'}{v(t,y)}{\rm dvol}_M\\
&\leq &||B_t^{-1}B_t'||_\infty t_0\int_Mt^{-n-1} |\alpha|^2 {\rm dvol}_M\\&&+2\left(\int_Mt^{-n-1}|\alpha|^2{\rm dvol}_M\right)^\frac{1}{2}\left(\int_Mt^{-n+1}|\alpha'|^2{\rm dvol}_M\right)^\frac{1}{2}\\&&+\bigg \Vert\frac{v(t,y)'}{v(t,y)}\bigg\Vert_\infty t_0\int_Mt^{-n-1}|\alpha|^2{\rm dvol}_M.
\end{eqnarray*}
If we denote by $M_1:=||B_t^{-1}B_t'||_\infty+||\frac{v(t,y)'}{v(t,y)}||_\infty$, 
then, after squaring both sides and taking $t_0$ small enough such that 
\begin{equation}\label{eq:smallt0}
    2-M_1t_0>0,
\end{equation} 
and, as a consequence, that $n - M_1t_0>0$, we get 
\eqref{eq:inequalityalphaalpha'magnetic}.

Now, we define the integral $$I:=\int_M t^{-n+1}(|d^\eta\alpha|^2+|\delta^\eta\alpha|^2){\rm dvol}_M.$$ 
We define $C_t:=B'_tB_t^{-1}+\frac{v'(t,y)}{v(t,y)}+i\eta_2$.
Using Equations \eqref{eq:exteriordiffmagnetic} and \eqref{eq:deltaalphamagnetic}, we get 
\begin{eqnarray}\label{eq:expressionimagnetic}
I&=&\int_M t^{-n+1}(|\alpha'|^2+|d^{\eta_1}\alpha_1|^2+|d^{\eta_1}\alpha_2|^2+\eta_2^2|\alpha_1|^2+|\delta^{\eta_1}_t\alpha_1|^2+|C_t\alpha_2|^2+|\delta^{\eta_1}_t\alpha_2|^2\nonumber\\&&-2\Re(\langle \alpha_1',d^{\eta_1}\alpha_2\rangle)+2\Re(\langle \alpha_1',i\eta_2\alpha_1\rangle)-2\Re(\langle d^{\eta_1}\alpha_2,i\eta_2\alpha_1\rangle)-2\Re(\langle \delta_t^{\eta_1}\alpha_1,C_t\alpha_2\rangle\nonumber\\&&-2\Re(\langle\delta_t^{\eta_1}\alpha_1,\alpha_2'\rangle)+2\Re(\langle C_t\alpha_2,\alpha_2'\rangle)) {\rm dvol}_M. 
\end{eqnarray}
We need to compute some of the cross terms. For this, we write 
\begin{eqnarray*} 
0&=&\int_0^{t_0} \frac{d}{dt}(t^{-n+1}\langle \alpha_1,d^{\eta_1}\alpha_2\rangle_{t} v(t,y))dt\\
&=&\int_0^{t_0} \frac{d}{dt}(t^{-n+1}\langle B_t\alpha_1,d^{\eta_1}\alpha_2\rangle_{0} v(t,y))dt\\
%&=&-(n-1)\int_0^{t_0} t^{-n}\langle B_t\alpha_1,d^{\eta_1}\alpha_2\rangle_{g_0} v(t,y)dt+\int_0^{t_0} t^{-n+1}\langle B_t'\alpha_1, d^{\eta_1}\alpha_2\rangle_{g_0}v(t,y)dt\\
%&&+\int_0^{t_0}t^{-n+1}\langle B_t\alpha_1', d^{\eta_1}\alpha_2\rangle_{g_0}v(t,y)dt+\int_0^{t_0}t^{-n+1}\langle B_t\alpha_1, d^{\eta_1}\alpha_2'\rangle_{g_0}v(t,y)dt\\&&+\int_0^{t_0}t^{-n+1}\langle B_t\alpha_1, i\eta_1'\wedge\alpha_2\rangle_{g_0}v(t,y)dt+\int_0^{t_0}t^{-n+1}\langle B_t\alpha_1, d^{\eta_1}\alpha_2\rangle_{g_0}v(t,y)'dt\\
&=&-(n-1)\int_0^{t_0} t^{-n}\langle \alpha_1,d^{\eta_1}\alpha_2\rangle_{t} v(t,y)dt+\int_0^{t_0} t^{-n+1}\langle B_t^{-1}B_t'\alpha_1, d^{\eta_1}\alpha_2\rangle_{t}v(t,y)dt\\
&&+\int_0^{t_0}t^{-n+1}\langle \alpha_1', d^{\eta_1}\alpha_2\rangle_{t}v(t,y)dt+\int_0^{t_0}t^{-n+1}\langle \delta_t^{\eta_1}\alpha_1, \alpha_2'\rangle_{t}v(t,y)dt\\&&+\int_0^{t_0}t^{-n+1}\langle \alpha_1, i\eta_1'\wedge\alpha_2\rangle_{t}v(t,y)dt+\int_0^{t_0}t^{-n+1}\langle \alpha_1, d^{\eta_1}\alpha_2\rangle_{t}v(t,y)'dt.
\end{eqnarray*}
Integrating over $(\partial M,g)$, we get that 
\begin{align*}
&\int_Mt^{-n+1}\langle \alpha_1', d^{\eta_1}\alpha_2\rangle {\rm dvol}_M+\int_Mt^{-n+1}\langle \delta_t^{\eta_1}\alpha_1, \alpha_2'\rangle {\rm dvol}_M
\\& \quad =(n-1)\int_M t^{-n}\langle \alpha_1,d^{\eta_1}\alpha_2\rangle {\rm dvol}_M-\int_M t^{-n+1}\langle B_t^{-1}B_t'\alpha_1, d^{\eta_1}\alpha_2\rangle {\rm dvol}_M
\\& \quad \quad -\int_M t^{-n+1}\langle \alpha_1, i\eta_1'\wedge\alpha_2\rangle {\rm dvol}_M-\int_M t^{-n+1}\langle \alpha_1, d^{\eta_1}\alpha_2\rangle \frac{v(t,y)'}{v(t,y)} {\rm dvol}_M. 
\end{align*} 
By replacing this last equality into the expression of $I$ in \eqref{eq:expressionimagnetic}, we get that 
\begin{eqnarray}\label{eq:estimateimagnetic}
I&=&\int_M t^{-n+1}(|\alpha'|^2+|d^{\eta_1}\alpha_1|^2+|d^{\eta_1}\alpha_2|^2+|\delta_t^{\eta_1}\alpha_1|^2+|C_t\alpha_2|^2+\eta_2^2|\alpha_1|^2+|\delta_t^{\eta_1}\alpha_2|^2){\rm dvol}_M\nonumber\\&&-2\int_M t^{-n+1}\Re(\langle \delta_t^{\eta_1}\alpha_1,C_t\alpha_2\rangle){\rm dvol}_M+2\int_M t^{-n+1}\Re(\langle C_t\alpha_2,\alpha_2'\rangle)) {\rm dvol}_M\nonumber\\&&-(n-1)\int_M t^{-n}\Re(\langle \alpha_1,d^{\eta_1}\alpha_2\rangle) {\rm dvol}_M-(n-1)\int_M t^{-n}\Re(\langle \delta_t^{\eta_1}\alpha_1,\alpha_2\rangle) {\rm dvol}_M\nonumber\\&&+2\int_M t^{-n+1}\Re(\langle B_t^{-1}B_t'\alpha_1, d^{\eta_1}\alpha_2\rangle) {\rm dvol}_M+2\int_M t^{-n+1}\Re(\langle \alpha_1, i\eta_1'\wedge\alpha_2\rangle) {\rm dvol}_M\nonumber\\&&
+2\int_M t^{-n+1}\Re(\langle \alpha_1, d^{\eta_1}\alpha_2\rangle) \frac{v(t,y)'}{v(t,y)} {\rm dvol}_M+2\int_M t^{-n+1}\Re(\langle\alpha_1',i\eta_2\alpha_1\rangle){\rm dvol}_M\nonumber\\
&&-2\int_M t^{-n+1}\Re(\langle d^{\eta_1}\alpha_2,i\eta_2\alpha_1\rangle){\rm dvol}_M\nonumber\\
&\geq &\int_M t^{-n+1}|\alpha'|^2{\rm dvol}_M+{\bf I}+{\bf II}+{\bf III},
\end{eqnarray}
where ${\bf I}, {\bf II}$ and ${\bf III}$ are given by
\begin{eqnarray*}
{\bf I}&=&\int_M t^{-n+1}|d^{\eta_1}\alpha_2|^2{\rm dvol}_M-(n-1)\int_M t^{-n}\Re(\langle \alpha_1,d^{\eta_1}\alpha_2\rangle) {\rm dvol}_M\\&&+2\int_M t^{-n+1}\Re(\langle B_t^{-1}B_t'\alpha_1, d^{\eta_1}\alpha_2\rangle) {\rm dvol}_M+2\int_M t^{-n+1}\Re(\langle \alpha_1, d^{\eta_1}\alpha_2\rangle) \frac{v(t,y)'}{v(t,y)} {\rm dvol}_M\\&&-2\int_M t^{-n+1}\Re(\langle d^{\eta_1}\alpha_2,i\eta_2\alpha_1\rangle){\rm dvol}_M,
\end{eqnarray*}
and 
\begin{eqnarray*}
{\bf II}&=&\int_M t^{-n+1}|\delta_t^{\eta_1}\alpha_1|^2{\rm dvol}_M-2\int_M t^{-n+1}\Re(\langle \delta_t^{\eta_1}\alpha_1,C_t\alpha_2\rangle){\rm dvol}_M\\&&-(n-1)\int_M t^{-n}\Re(\langle \delta_t^{\eta_1}\alpha_1,\alpha_2\rangle) {\rm dvol}_M,
\end{eqnarray*}
and 
\begin{eqnarray*}
{\bf III}&=&2\int_M t^{-n+1}\Re(\langle C_t\alpha_2,\alpha_2'\rangle) {\rm dvol}_M+2\int_M t^{-n+1}\Re(\langle \alpha_1,i\eta_1'\wedge\alpha_2\rangle) {\rm dvol}_M\\&&+2\int_M t^{-n+1}\Re(\langle \alpha'_1,i\eta_2\alpha_1\rangle) {\rm dvol}_M.
\end{eqnarray*}
Now, we estimate ${\bf I}, {\bf II}$ and ${\bf III}$ from below as follows. We take $M_2:=||B_t^{-1}B_t'||_\infty+||\frac{v(t,y)'} {v(t,y)}||_\infty+||\eta_2||_\infty$ and we have
\begin{eqnarray*}
 {\bf I}&\geq& \int_M (t^{-n+1}|d^{\eta_1}\alpha_2|^2-(n-1)t^{-n} |\alpha_1||d^{\eta_1}\alpha_2|-2M_2t_0t^{-n}|\alpha_1||d^{\eta_1}\alpha_2|) {\rm dvol}_M\\
 &=&\int_M \left((t^{\frac{-n+1}{2}}|d^{\eta_1}\alpha_2|)^2-2t^{\frac{-n+1}{2}}|d^{\eta_1}\alpha_2|\left(\frac{n-1+2M_2t_0}{2}\right) t^{\frac{-n-1}{2}} |\alpha_1|\right) {\rm dvol}_M\\
 &\geq &-\left(\frac{n-1+2M_2t_0}{2}\right)^2\int_M t^{-n-1}|\alpha_1|^2{\rm dvol}_M,
 \end{eqnarray*}
where we used the inequality $a^2-2ab\geq -b^2$ in the final step. 

In the same way, we estimate ${\bf II}$ by setting $M_3=||C_t||_\infty$ and we have
\begin{eqnarray*}
 {\bf II}&\geq& \int_M (t^{-n+1}|\delta_t^{\eta_1}\alpha_1|^2-2M_3t_0t^{-n} |\delta_t^{\eta_1}\alpha_1||\alpha_2|-(n-1)t^{-n}|\delta_t^{\eta_1}\alpha_1||\alpha_2|) {\rm dvol}_M\\
 &\geq &-\left(\frac{n-1+2M_3t_0}{2}\right)^2\int_M t^{-n-1}|\alpha_2|^2{\rm dvol}_M.
 \end{eqnarray*}
 
We are now left to estimate ${\bf III}$. For this we write 
\begin{eqnarray*}
 {\bf III}&=& 2\int_M \Re(\langle t^{\frac{-n-1}{2}}C_t\alpha_2,t^{\frac{-n+3}{2}}\alpha_2'\rangle)) {\rm dvol}_M\\&&+2\int_M t^2\Re(\langle t^{\frac{-n-1}{2}}\alpha_1,it^{\frac{-n-1}{2}}\eta_1'\wedge\alpha_2\rangle) {\rm dvol}_M\\&&+2\int_M \Re(\langle t^{\frac{-n+3}{2}} \alpha'_1,it^{\frac{-n-1}{2}}\eta_2\alpha_1\rangle) {\rm dvol}_M\\
 &\geq &-2M_3\left(\int_M t^{-n-1}|\alpha_2|^2{\rm dvol}_M\right)^\frac{1}{2}\left(\int_M t^{-n+3}|\alpha_2'|^2{\rm dvol}_M\right)^\frac{1}{2}\\&&-2||\eta_1'||_\infty t_0^2\int_M t^{-n-1}|\alpha|^2{\rm dvol}_M\\
 &&-2||\eta_2||_\infty\left(\int_M t^{-n-1}|\alpha_1|^2{\rm dvol}_M\right)^\frac{1}{2}\left(\int_M t^{-n+3}|\alpha_1'|^2{\rm dvol}_M\right)^\frac{1}{2}\\
&\stackrel{\eqref{eq:inequalityalphaalpha'magnetic}}{\geq }&-\frac{4M_4}{n-M_1t_0}\left(\int_M t^{-n+1}|\alpha'|^2{\rm dvol}_M\right)^\frac{1}{2}\left(\int_M t^{-n+3}|\alpha'|^2{\rm dvol}_M\right)^\frac{1}{2}\\&&-\frac{8||\eta_1'||_\infty t_0^2}{(n-M_1t_0)^2} \int_M t^{-n+1}|\alpha'|^2{\rm dvol}_M\\
&\geq &-\left(\frac{4M_4t_0}{n-M_1t_0}+\frac{8||\eta_1'||_\infty t_0^2}{(n-M_1t_0)^2}\right)\int_M t^{-n+1}|\alpha'|^2{\rm dvol}_M,
\end{eqnarray*}
where $M_4=M_3+||\eta_2||_\infty$. 

Plugging these estimates into \eqref{eq:estimateimagnetic}, for any $t_0$ such that \eqref{eq:smallt0} holds, we obtain that
\begin{eqnarray*}
I&\geq& \left(1-\frac{4M_4t_0}{n-M_1t_0}-\frac{8||\eta_1'||_\infty t_0^2}{(n-M_1t_0)^2}\right)\int_M t^{-n+1}|\alpha'|^2{\rm dvol}_M\\&&-\left(\frac{n-1+2M_2t_0}{2}\right)^2\int_M t^{-n-1}|\alpha_1|^2{\rm dvol}_M\\&&-\left(\frac{n-1+2M_3t_0}{2}\right)^2\int_M t^{-n-1}|\alpha_2|^2{\rm dvol}_M\\
&\stackrel{\eqref{eq:inequalityalphaalpha'magnetic}}{\geq }&\left(1-\frac{4M_4t_0}{n-M_1t_0}-\frac{8||\eta_1'||_\infty t_0^2}{(n-M_1t_0)^2}\right)\left(\frac{n-M_1t_0}{2}\right)^2\int_M t^{-n-1}(|\alpha_1|^2+|\alpha_2|^2){\rm dvol}_M\\&&-\left(\frac{n-1+2M_2t_0}{2}\right)^2\int_M t^{-n-1}|\alpha_1|^2{\rm dvol}_M\\&&-\left(\frac{n-1+2M_3t_0}{2}\right)^2\int_M t^{-n-1}|\alpha_2|^2{\rm dvol}_M.\\
&\geq&\left(1-\frac{4M_4t_0}{n-M_1t_0}-\frac{8||\eta_1'||_\infty t_0^2}{(n-M_1t_0)^2}\right)\left(\frac{n-M_1t_0}{2}\right)^2\int_M t^{-n-1}(|\alpha_1|^2+|\alpha_2|^2){\rm dvol}_M\\&&-\left(\frac{n-1+2(M_2+M_3)t_0}{2}\right)^2\int_M t^{-n-1}(|\alpha_1|^2+|\alpha_2|^2){\rm dvol}_M.\\
\end{eqnarray*}
Now
\begin{eqnarray*}
I&\geq& \liminf_{t_0\to 0}\left[\left(1-\frac{4M_4t_0}{n-M_1t_0}-\frac{8||\eta_1'||_\infty t_0^2}{(n-M_1t_0)^2}\right)\left(\frac{n-M_1t_0}{2}\right)^2\int_M t^{-n-1}(|\alpha_1|^2+|\alpha_2|^2){\rm dvol}_M\right.\\&&-\left.\left(\frac{n-1+2(M_2+M_3)t_0}{2}\right)^2\int_M t^{-n-1}(|\alpha_1|^2+|\alpha_2|^2){\rm dvol}_M\right]\\
&=& \frac{2n-1}{4}\int_M t^{-n-1}|\alpha|^2 {\rm dvol}_M \\
&\geq& \frac{3}{4} \int_M t^{-n-1}|\alpha|^2 {\rm dvol}_M.
\end{eqnarray*}
So, by continuity of the expression, there exists $t_0$ such that
$$I \geq \frac{1}{4}\int_M t^{-n-1}|\alpha|^2 {\rm dvol}_M$$
as required.
\end{proof}
We apply Lemma \ref{lem:magneticineq} to the form $\alpha=f\omega$, where $f$ is a real-valued smooth function which is zero for $t>t_0$ and is equal to $1$ on $[0,t_1]$ for some $t_1$ such that $0<t_1<t_0$. Clearly the form $\alpha$ satisfies the assumption of the lemma and, hence, we can apply the inequality to the form $\alpha$.
Since $\omega$ is magnetic harmonic, we get $|d^\eta\alpha|^2+|\delta^\eta\alpha|^2=|df\wedge \omega+fd^\eta\omega|^2+|\nabla f\lrcorner\omega+f\delta^\eta\omega|^2= |\nabla f|^2|\omega|^2$, and hence that
\begin{eqnarray*}
t_1^n\int_0^{t_1}\int_{\partial M} t^{-n-1}|\omega|^2v(t,y)dt\,{\rm dvol}_{\partial M}&\leq& t_1^n\int_M t^{-n-1}|\omega|^2 {\rm dvol}_M\\
&\leq & 4 t_1^n\int_0^{t_0}\int_{\partial M} t^{-n+1}|\nabla f|^2|\omega|^2v(t,y)\,dt{\rm dvol}_{\partial M}\\
&=&4\int_{t_1}^{t_0}\int_{\partial M} \left(\frac{t_1}{t}\right)^{n}t|\nabla f|^2|\omega|^2v(t,y)dt\,{\rm dvol}_{\partial M}\\
&\leq&4\int_{t_1}^{t_0}\int_{\partial M} t|\nabla f|^2|\omega|^2v(t,y)dt\,{\rm dvol}_{\partial M}\\&=&C<\infty
\end{eqnarray*}
where $C$ does not depend on $n$. Hence, the l.h.s. is bounded in $n$ but this cannot be true unless $\omega$ vanishes on $[0,t_1]$. Indeed, assume that $||\omega||_{L^2(g_t)}^2>A$ at some point $t_2\leq t_1$, then by continuity, we have $||\omega||_{L^2(g_t)}^2>\frac{A}{2}$ on some interval $(t_2-\varepsilon, t_2+\varepsilon)$ for some small $\varepsilon>0$. Hence, we write 
\begin{eqnarray*}
C\geq t_1^n\int_0^{t_1} \int_{\partial M}t^{-n-1}|\omega|^2v(t,y)dt {\rm dvol}_N&\geq& t_1^n\frac{A}{2}\int_{t_2-\varepsilon}^{t_2+\varepsilon} t^{-n-1}dt\\
&\geq &\frac{1}{n}\left(\left(\frac{t_1}{t_2-\varepsilon}\right)^n-\left(\frac{t_1}{t_2+\varepsilon}\right)^n\right) \to \infty,
\end{eqnarray*}
as $n\to \infty$ which is a contradiction. Hence $\omega$ vanishes on a tubular neighbourhood. 
\end{proof}
{\bf Step 3:} We now construct a double copy $\widetilde{M}$ of the manifold $M$, and we extend the metric in a smooth way to keep the form $\omega$ magnetic harmonic on $\widetilde{M}$ and such that it vanishes on some open, non-empty subset $U$. Now, one can easily show that we have pointwise $|d\omega|^2+|\delta\omega|^2=|\eta|^2|\omega|^2$ on $\widetilde{M}$. Hence by using \cite[Thm. 3.4.3]{Sch:95}, we deduce that $\omega=0$ on $M$. This ends the proof of Theorem  \ref{thm:dirichletmagnetic}.
\hfill$\square$

\subsection{Existence of a solution}\label{ss:exist}
In this section, we prove the existence of solutions of the boundary value problem in Theorem \ref{thm:magneticboundprob}. For the proof, we need to introduce and show several intermediate results. We follow \cite[Chap. 2]{Sch:95} closely. 

In what follows, we consider a compact manifold $(M^m,g)$ of dimension $m$ with smooth boundary $\partial M$ and denote by $L^2=L^2(\Omega^k(M,\mathbb{C}))$ and by $H^s=H^s(\Omega^k(M,\mathbb{C}) )$ the Sobolev spaces of the set $\Omega^k(M,\mathbb{C})$ for $0\leq k\leq m$, that is the completion of  $\Omega^k(M,\mathbb{C})$ with respect to the norm $||\omega||^2_{H^s}=||\omega||^2_{H^{s-1}}+||\nabla\omega||^2_{H^{s-1}}$, where $\nabla$ is the Levi-Civita connection on $M$. As in \cite{Sch:95}, we define the magnetic absolute and relative cohomologies. For this, we denote by $\iota:\partial M\to M$ the embedding of $\partial M$ into $M$ and define 
$$H^1(\Omega^k_{\rm rel}(M, \mathbb{C})):=\{\omega\in H^1(\Omega^k(M,\mathbb{C}))|\,\, \iota^*\omega=0\,\, \text{on}\,\, \partial M\},$$
and 
$$H^1(\Omega^k_{\rm abs}(M, \mathbb{C})):=\{\omega\in H^1(\Omega^k(M,\mathbb{C}))|\,\, \nu\lrcorner\omega=0\,\, \text{on}\,\, \partial M\}.$$
The magnetic relative and absolute cohomologies are then defined by 
$$\mathcal{H}^{k,\eta}_{\rm rel}(M):=\{\omega\in H^1(\Omega^k(M,\mathbb{C}))|\,\, d^\eta \omega=0,\, \delta^\eta\omega=0, \,\iota^*\omega=0\},$$
and 
$$\mathcal{H}^{k,\eta}_{\rm abs}(M):=\{\omega\in H^1(\Omega^k(M,\mathbb{C}))|\,\, d^\eta \omega=0,\, \delta^\eta\omega=0, \,\nu\lrcorner\omega=0\}.$$
It was shown in \cite[Lem. 3.1]{EGHP:23} that the Hodge star operator $*$ interchanges (up to a sign) $d^\eta$ and $\delta^\eta$ as well as the boundary conditions. Therefore, we deduce that $\mathcal{H}^{k,\eta}_{\rm rel}(M)\simeq\mathcal{H}^{m-k,\eta}_{\rm abs}(M)$. Note that for manifolds with boundary, there is a subtle distinction between harmonic fields and harmonic forms (see \cite[Rem. 2, p. 67]{Sch:95}). In the following, we will prove that these cohomologies are of finite dimension and consist of smooth fields. For this, we define the magnetic Dirichlet integral 
$\mathcal{D}^\eta(\cdot,\cdot)$ as
\begin{align*}
\mathcal{D}^\eta:H^1\times H^1&\longrightarrow\mathbb{C}\\ (\phi,\psi)& \longmapsto\int_M (\langle d^\eta\phi,d^\eta\psi\rangle+\langle \delta^\eta\phi,\delta^\eta\psi\rangle){\rm dvol}_M.
\end{align*}
\begin{lemma}[{Magnetic Gaffney inequality, c.f. \cite[Cor. 2.1.6]{Sch:95} for the non-magnetic version}]
\label{lem:gaffny}
Let $(M^m,g)$ be a compact manifold of dimension $m$ with smooth boundary and let $\omega\in H^1(\Omega^k_{\rm rel}(M, \mathbb{C}))$. Then there exists a constant $C>0$ depending only on the geometry of $M$ and the magnetic potential $\eta$, such that  
$$||\omega||_{H^1}^2\leq C (\mathcal{D}^\eta(\omega,\omega)+||\omega||^2_{L^2}).$$
\end{lemma}
\begin{proof} For any $\omega\in H^1$, we write 
\begin{equation*}
|d\omega|^2=|d\omega+i\eta\wedge\omega-i\eta\wedge\omega|^2
=|d^\eta\omega-i\eta\wedge\omega|^2
\leq 2(|d^\eta\omega|^2+|\eta\wedge\omega|^2).
\end{equation*}
In the same way, we can prove that $|\delta\omega|^2\leq 2(|\delta^\eta\omega|^2+|\eta\lrcorner\omega|^2)$. Therefore, we can deduce that the inequality 
$$\mathcal{D}(\omega,\omega)\leq 2\mathcal{D}^\eta(\omega,\omega)+2||\eta||_\infty^2||\omega||_{L^2}^2$$ 
holds for any $\omega\in H^1$. Here we have used the fact $|\eta\wedge\omega|^2+|\eta\lrcorner\omega|^2=|\eta|^2|\omega|^2$ which holds pointwise. Now, the result follows by the standard Gaffney inequality \cite[Cor. 2.1.6]{Sch:95}.     
\end{proof}

\begin{remark} A direct computation shows that 
\begin{equation}\label{eq:relationdetad}
\mathcal{D}^\eta(\omega,\omega)\leq 2\mathcal{D}(\omega,\omega)+2||\eta||_\infty^2||\omega||_{L^2}^2
\end{equation}
for any $\omega\in H^1$.
\end{remark} 

\begin{theorem}[{c.f. \cite[Thm 2.2.2]{Sch:95} for the non-magnetic version}]
Let $(M^m,g)$ be a compact Riemannian manifold of dimension $m$ with smooth boundary. The space $\mathcal{H}^{k,\eta}_{\rm rel}(M)$ is finite dimensional. 
\end{theorem}
\begin{proof}
We follow the same proof as in \cite[Thm. 2.2.2]{Sch:95}. Let $\omega$ be an element in $\mathcal{H}^{k,\eta}_{\rm rel }(M)$, the magnetic Gaffney inequality gives in particular $||\omega||_{H^1}^2\leq C ||\omega||^2_{L^2}.$ Therefore, the $H^1$ and $L^2$-norms are equivalent on $\mathcal{H}^{k,\eta}_{\rm rel}(M)$. Let $D_1$ denote the unit disk in $\mathcal{H}^{k,\eta}_{\rm rel}(M)$, that is 
$$D_1=\{\omega\in \mathcal{H}^{k,\eta}_{\rm rel}(M)|\,\, ||\omega||_{H^1}\leq 1\}.$$
Since $D_1$ is closed with respect to the $H^1$-topology, it is also closed with respect to the $L^2$-topology. By the Rellich lemma (see, e.g., \cite[Thm 1.3.3(c)]{Sch:95}), the embedding $H^1\hookrightarrow L^2$ is compact and, thus, $D_1$ is relatively compact in $L^2$. But, since it is closed in the $L^2$-topology, it is then compact. Hence the unit disk is compact in $\mathcal{H}^{k,\eta}_{\rm rel}(M)$ and, therefore, of finite dimension. 
\end{proof}

In what follows, we will show the following theorem.
\begin{theorem}\label{thm:smoothcohomology}
Any element in the magnetic relative cohomology group $\mathcal{H}^{k,\eta}_{\rm rel}(M)$ is smooth.  
\end{theorem}
To prove this theorem, we need several technical lemmas and propositions. We follow the steps done in \cite[Sect. 2.2 and 2.3]{Sch:95} closely. We begin with the following definition. 
\begin{definition}[Magnetic Dirichlet potential] 
Let $\theta\in L^2$. If there exists $\phi_D\in H^1(\Omega^k_{\rm rel}(M,\mathbb{C}))$ such that
\begin{equation}\label{eq:weakboundaryproblem}
\mathcal{D}^\eta(\phi_D, \hat\xi)=(\theta,\hat\xi)_{L^2}
\end{equation}
for all $\hat\xi\in H^1(\Omega^k_{\rm rel}(M,\mathbb{C}))$, then we call $\phi_D$ a magnetic Dirichlet potential associated with $\theta$.
\end{definition}
Note that magnetic Dirichlet potentials associated with $\theta$ may not exist and are generally not unique, since $\phi_D+\omega$ with $\omega\in \mathcal{H}^{k,\eta}_{\rm rel}(M)$ is again a magnetic Dirichlet potential associated with $\theta$. As we have shown that $\mathcal{H}^{k,\eta}_{\rm rel}(M)$ is finite dimensional, we have the orthogonal decomposition $L^2=\mathcal{H}^{k,\eta}_{\rm rel}(M)\oplus (\mathcal{H}^{k,\eta}_{\rm rel}(M))^{\perp_{L^2}}$.  Now, we set 
$$\widehat{\mathcal{H}^{k,\eta}_{\rm rel}(M)}:=H^1(\Omega^k_{\rm rel}(M, \mathbb{C}))\cap (\mathcal{H}^{k,\eta}_{\rm rel}(M))^{\perp_{L^2}}.$$ In the following, when $\theta$ belongs to some subspace of $L^2$, we show the existence and uniqueness of a Dirichlet potential on $(\mathcal{H}^{k,\eta}_{\rm rel}(M))^{\perp_{L^2}}$.

We prove the following proposition which generalizes \cite[Prop. 2.2.3 and Th. 2.2.4]{Sch:95} to the magnetic setting, 

\begin{proposition}[{c.f. \cite[Prop. 2.2.3, Thm 2.2.4]{Sch:95} for the non-magnetic versions}] \label{lem:dirichletpotential} There exist constants $C, c >0$ such that, for any $\omega\in \widehat{\mathcal{H}^{k,\eta}_{\rm rel}(M)}$,
\begin{equation}\label{eq:ellipdeta}
c||\omega||^2_{H^1}\leq \mathcal{D}^\eta(\omega,\omega)\leq C||\omega||^2_{H^1}.
\end{equation}
In particular, for any $\theta\in (\mathcal{H}^{k,\eta}_{\rm rel}(M))^{\perp_{L^2}}$, there exists a unique differential form $\phi_D\in \widehat{\mathcal{H}^{k,\eta}_{\rm rel}(M)}$ satisfying \eqref{eq:weakboundaryproblem}.
\end{proposition} 
Before proving this proposition, we need the following well-known result.
We provide a proof for completeness.
\begin{lemma} \label{lem:generalstathilbert}
    Let $H$ be any complex Hilbert space with a Hermitian product $\langle\cdot,\cdot\rangle_H$ and let $Q:H\times H\to\mathbb{C}$ be a Hermitian form such that $Q$ is continuous and nonnegative, that is $Q(v,v)\geq 0$ for any $v\in H$. If a sequence $(\omega_j)_j$ converges weakly to $\theta$ in $H$, then $\liminf_{j \to \infty} Q(\omega_j,\omega_j)\geq Q(\theta,\theta)$.
\end{lemma}
\begin{proof} For any $v\in H$, the form $Q(\cdot,v)$ is linear and continuous, and hence, is an element in $H'$. By the Riesz representation theorem, there exists a unique element that we denote by $Q^*v$ in $H$ such that 
$$Q(u,v)=\langle u,Q^*v\rangle_H,$$
for any $u\in H$. Now, for each $j$, we have that
\begin{eqnarray*}
Q(\omega_j,\omega_j)&=&Q(\omega_j,\theta)+Q(\omega_j-\theta,\omega_j-\theta)+Q(\theta,\omega_j-\theta)\\
&=&\langle\omega_j, Q^*\theta\rangle_H+Q(\omega_j-\theta,\omega_j-\theta)+\overline{\langle \omega_j-\theta,Q^*\theta\rangle}_H\\
&\geq &\langle\omega_j, Q^*\theta\rangle_H+\overline{\langle \omega_j-\theta,Q^*\theta\rangle}_H.
\end{eqnarray*}
By taking the limit as $j\to \infty$, we get that $$\liminf_{j \to \infty} Q(\omega_j,\omega_j)\geq \langle\theta, Q^*\theta\rangle_H=Q(\theta,\theta),$$ 
which is the statement of the lemma.
\end{proof}
\begin{proof}[Proof of Proposition \ref{lem:dirichletpotential}] The right-hand side. of Inequality \eqref{eq:ellipdeta} is a direct consequence of \eqref{eq:relationdetad} and the continuity of $d$ and $\delta$ seen as operators from $H^1$ to $L^2$. 

To show the left-hand side, we let $S_1$ denote the unit $L^2$-sphere in $\widehat{\mathcal{H}^{k,\eta}_{\rm rel}(M)}$, i.e. 
$$S_1=\{\omega\in \widehat{\mathcal{H}^{k,\eta}_{\rm rel}(M)}|\,\, ||\omega||_{L^2}=1 \}.$$ Let $(\omega_i)_i$ be a sequence in $S_1$ such that $\mathcal{D}^\eta(\omega_i,\omega_i)\to \mathop{\rm inf}\limits_{\omega\in S_1}(\mathcal{D}^\eta(\omega,\omega))$. By Lemma \ref{lem:gaffny}, the sequence $(\omega_i)_i$ is $H^1$-bounded, since $(\mathcal{D}^\eta(\omega_i,\omega_i))_i$ is bounded as a convergent sequence.  
Now, the uniform boundedness principle for the Hilbert space $\widehat{\mathcal{H}^{k,\eta}_{\rm rel}(M)}$, and \cite[Cor. 1.5.2]{Sch:95} applied to the Rellich embedding $ \widehat{\mathcal{H}^{k,\eta}_{\rm rel}(M)}\hookrightarrow L^2$ gives that $(\omega_i)_i$ has a subsequence which is weakly convergent in $\widehat{\mathcal{H}^{k,\eta}_{\rm rel}(M)}$ and strongly convergent in $L^2$ with limit $\theta$ in  $\widehat{\mathcal{H}^{k,\eta}_{\rm rel}(M)}$. 
Hence, $||\theta||_{L^2}=1$ (since $||\omega_i||_{L^2}=1$ and the continuity of the norm). Thus, $\theta\neq 0$ and from the definition of $\widehat{\mathcal{H}^{k,\eta}_{\rm rel}(M)}$, we get that $\mathcal{D}^\eta(\theta,\theta)>0$. Now, we apply Lemma \ref{lem:generalstathilbert} to the complex Hilbert space $\widehat{\mathcal{H}^{k,\eta}_{\rm rel}(M)}$ and to the Hermitian form $\mathcal{D}^\eta$ to get that $\mathop{\rm inf}\limits_{\omega\in S_1}(\mathcal{D}^\eta(\omega,\omega))\geq \mathcal{D}^\eta(\theta,\theta)>0$. Hence we deduce that for any $\omega\in \widehat{\mathcal{H}^{k,\eta}_{\rm rel}(M)}$, the inequality
   $$\mathcal{D}^\eta(\omega,\omega)\geq \mathcal{D}^\eta(\theta,\theta)||\omega||_{L^2}^2$$
holds. Thus, by the magnetic Gaffney inequality, we have that
$$||\omega||_{H^1}^2\leq C \left(1+\frac{1}{\mathcal{D}^\eta(\theta,\theta)}\right)\mathcal{D}^\eta(\omega,\omega),$$
giving the left-hand side of \eqref{eq:ellipdeta}. 

Now, we will prove the existence and uniqueness of a Dirichlet potential in $\widehat{\mathcal{H}^{k,\eta}_{\rm rel}(M)}$. The form $\mathcal{D}^\eta$ is a  Hermitian form in the complex Hilbert space $\widehat{\mathcal{H}^{k,\eta}_{\rm rel}(M)}$, which is $H^1$-elliptic by Inequality \eqref{eq:ellipdeta}. Therefore, by the Lax-Milgram Theorem \cite[Cor. 1.5.10]{Sch:95}, for any $\theta\in (\mathcal{H}^{k,\eta}_{\rm rel}(M))^{\perp_{L^2}}$, there exists a differential form $\phi_D\in \widehat{\mathcal{H}^{k,\eta}_{\rm rel}(M)}$ such that  
$$\mathcal{D}^\eta(\phi_D, \hat\xi)=(\theta,\hat\xi)_{L^2}$$
for any $\hat\xi\in \widehat{\mathcal{H}^{k,\eta}_{\rm rel}(M)}$. 

Notice here that this equality still holds for any $\hat\xi\in H^1(\Omega^k_{\rm rel}(M, \mathbb{C}))$: Indeed, with respect to the orthogonal decomposition $$L^2=\mathcal{H}^{k,\eta}_{\rm rel}(M)\oplus (\mathcal{H}^{k,\eta}_{\rm rel}(M))^{\perp_{L^2}},$$
we can write any $\hat\xi\in H^1(\Omega^k_{\rm rel}(M, \mathbb{C}))$ as $\hat\xi=\hat\xi_1+\hat\xi_2$ where $\hat\xi_1\in \mathcal{H}^{k,\eta}_{\rm rel}(M) $ and $\hat\xi_2\in (\mathcal{H}^{k,\eta}_{\rm rel}(M))^{\perp_{L^2}}\cap H^1(\Omega^k_{\rm rel}(M, \mathbb{C}))=\widehat{\mathcal{H}^{k,\eta}_{\rm rel}(M)}$. Then, we compute
\begin{eqnarray*}
 \mathcal{D}^\eta(\phi_D, \hat\xi)&=&\mathcal{D}^\eta(\phi_D, \hat\xi_1)+\mathcal{D}^\eta(\phi_D, \hat\xi_2)\\
 &=&\mathcal{D}^\eta(\phi_D, \hat\xi_2)\\
 &=&(\theta,\hat\xi_2)_{L^2}=(\theta,\hat\xi)_{L^2},
\end{eqnarray*}
because $\theta$ is an element in $(\mathcal{H}^{k,\eta}_{\rm rel}(M))^{\perp_{L^2}}$. 

Finally, we show uniqueness. Let $\phi_D$ and $\phi_D'$ be two Dirichlet potentials associated with $\theta$. Hence, we have  $\mathcal{D}^\eta(\phi_D-\phi_D',\hat\xi)=0$ for any $\hat\xi\in H^1(\Omega^k_{\rm rel}(M, \mathbb{C}))$. Hence taking $\hat\xi=\phi_D-\phi_D'$, we get that $\mathcal{D}^\eta(\hat\xi,\hat\xi)=0$, which means that $\hat\xi=\phi_D-\phi_D'\in \mathcal{H}^{k,\eta}_{\rm rel}(M)$. But since $\phi_D-\phi_D'$ is an element in $(\mathcal{H}^{k,\eta}_{\rm rel}(M))^{\perp_{L^2}}$, we deduce that $\phi_D=\phi_D'$.  
\end{proof}
\begin{remark} We point out that if $\phi_D$ is associated to $\theta$ and $\theta'$, we get that $(\theta-\theta',\hat\xi)_{L^2}=0$ for any $\hat\xi\in H^1(\Omega_{\rm rel}^k(M))$. By the density of $H^1(\Omega_{\rm rel}^k(M))$ in $L^2(\Omega_{\rm rel}^k(M))$, we get that $\theta=\theta'$.
\end{remark}

In what follows, we will show that any solution of \eqref{eq:weakboundaryproblem} associated to some $\theta\in L^2$ is in fact in the Sobolev space $H^2$. 
We have the following lemma. 
\begin{lemma}[{c.f. \cite[Lem. 2.3.1]{Sch:95} for the non-magnetic version}] \label{eq:estimateh1dirichlet} Let $X$ be any real vector field with compact support on $M$ such that $g(X,\nu)=0$ on $\partial M$ and let $\eta$ be a magnetic potential on $M$. We denote by $\psi_t^X$ the flow of $X$ restricted to $t\in [-1,1]\setminus \{0\}$, and by $\Sigma_t^X$ the map given by  $\Sigma_t^X\phi:=\frac{1}{t}((\psi_t^X)^*\phi-\phi)$ for any $\phi$. Then, we have the following.
\begin{enumerate} 
\item \label{part1} There exists a constant $C_1$ (independent of $t$) such that 
\begin{equation}\label{eq:sigmal2}
||\Sigma_t^X\phi||_{L^2}\leq C_1||\phi||_{H^1},
\end{equation}
for any $\phi\in H^1$.
\item  \label{part2} If $\phi_1,\phi_2\in H^1$, then there exists a constant $C_2$ (independent of $t$) such that 
\begin{equation}\label{eq:detesigma}
|\mathcal{D}^\eta(\Sigma_t^X\phi_1,\phi_2)|\leq |\mathcal{D}^\eta(\phi_1,\Sigma_{-t}^X\phi_2)|+C_2||\phi_1||_{H^1}||\phi_2||_{H^1},
\end{equation}
 
\item \label{part3} If $\phi_D\in H^1(\Omega^k_{\rm rel}(M,\mathbb{C}))$ is a magnetic Dirichlet potential associated with some $\theta\in L^2$, i.e. satisfying \eqref{eq:weakboundaryproblem}, then there exists a constant $\hat C$, that depends only on $\eta$ and $\phi_D$ but is independent of $t$, such that 
$$||\Sigma_t^X\phi_D||_{H^1}\leq \hat C.$$
\end{enumerate}
\end{lemma}

\begin{proof}
Note that Part \ref{part1} of the lemma is independent of the magnetic potential $\eta$, and is proven in \cite[Lem. 2.3.1, p. 74]{Sch:95}. 

For Part \ref{part2}, we define the bundle endomorphism $\widetilde\Sigma_t:\Omega^k(M,\mathbb{C})\to \Omega^k(M,\mathbb{C})$ by 
\begin{equation}\label{eq:differencestar}
*(\widetilde\Sigma_t^X\omega)=\Sigma_t^X(*\omega)-*(\Sigma_t^X\omega),
\end{equation}
for $\omega\in \Omega^k(M,\mathbb{C})$. Here $*$ is the Hodge star operator on $(M,g)$. It is shown in \cite[Lem. 1.2.4]{Sch:95} that $\widetilde\Sigma_t^X$ is a smooth bundle endomorphism with compact support. Consequently, we have  $||\widetilde\Sigma_t^X\omega||_{H^s}\leq c ||\omega||_{H^s}$ for any $t\in [-1,1]$.  Now, for any two complex forms $\omega_1,\omega_2$, 
we use \eqref{eq:differencestar} to compute
\begin{eqnarray*}
\Sigma_t^X(\omega_1\wedge *\overline\omega_2)-\Sigma_t^X\omega_1\wedge *\overline\omega_2&=&\frac{1}{t}(\psi_t^X)^*(\omega_1\wedge *\overline\omega_2)-\frac{1}{t}(\psi_t^X)^*(\omega_1)\wedge *\overline\omega_2\\
&=&(\psi_t^X)^*(\omega_1)\wedge \Sigma_t^X(*\overline\omega_2)\\
&=&(\psi_t^X)^*(\omega_1)\wedge (\psi_t^X)^*\Sigma_{-t}^X(*\overline\omega_2)\\
&\stackrel{\eqref{eq:differencestar}}{=}&(\psi_t^X)^*(\omega_1\wedge *(\widetilde\Sigma_{-t}^X\overline\omega_2))+(\psi_t^X)^*(\omega_1\wedge *(\Sigma_{-t}^X\omega_2)).
\end{eqnarray*}
In the third equality, we use the fact that $(\psi_t^X)^*\Sigma_{-t}^X=\Sigma_t^X$.
We therefore have (as in \cite[p. 75]{Sch:95}) that
\begin{equation}\label{eq:identitysigmax}
\Sigma_t^X(\omega_1\wedge *\overline\omega_2)=\Sigma_t^X\omega_1\wedge *\overline\omega_2+(\psi_t^X)^*(\omega_1\wedge *(\Sigma_{-t}^X\overline\omega_2))+(\psi_t^X)^*(\omega_1\wedge *(\widetilde\Sigma_{-t}^X\overline\omega_2)).
\end{equation}
Integrating \eqref{eq:identitysigmax} over $M$, using $\int_M(\psi_t^X)^*\beta=\int_M\beta$ for any $\beta\in \Omega^m(M,\mathbb{C})$, and observing the left-hand side of \eqref{eq:identitysigmax} has vanishing integral, yields 
\begin{equation}\label{eq:integralsigmat}
    \int_M \langle \Sigma_t^X\omega_1,\omega_2\rangle {\rm dvol}_M+\int_M \langle \omega_1,\Sigma_{-t}^X\omega_2\rangle {\rm dvol}_M+\int_M \langle \omega_1,\widetilde\Sigma_{-t}^X\omega_2\rangle {\rm dvol}_M=0.
\end{equation}
Next, we use the following identity 
\begin{eqnarray}\label{eq:commutatorsigma}
\Sigma_t^X(d^\eta\cdot)&=&\frac{1}{t}((\psi_t^X)^*(d^\eta\cdot)-d^\eta\cdot)\nonumber\\
&=&\frac{1}{t}(d((\psi_t^X)^*\cdot)+i(\psi_t^X)^*\eta\wedge (\psi_t^X)^*\cdot-d^\eta\cdot)\nonumber\\
&=&\frac{1}{t}(d^\eta((\psi_t^X)^*\cdot)-i\eta\wedge (\psi_t^X)^*\cdot+i(\psi_t^X)^*\eta\wedge (\psi_t^X)^*\cdot-d^\eta\cdot)\nonumber\\
&=&d^\eta(\Sigma_t^X\cdot)+i(\Sigma_t^X\eta)\wedge (\psi_t^X)^*\cdot.
\end{eqnarray}
We take $\omega_1=d^\eta\phi_1$ and $\omega_2=d^\eta\phi_2$ in \eqref{eq:integralsigmat} and use \eqref{eq:commutatorsigma} to get
\begin{eqnarray}\label{eq:detasigma}
\int_M \langle d^\eta(\Sigma_t^X\phi_1),d^\eta\phi_2\rangle {\rm dvol}_M+i\int_M \langle \Sigma_t^X\eta\wedge(\psi_t^X)^*\phi_1,d^\eta\phi_2\rangle {\rm dvol}_M\nonumber\\+\int_M \langle d^\eta\phi_1,d^\eta(\Sigma_{-t}^X\phi_2)\rangle {\rm dvol}_M-i\int_M\langle d^\eta\phi_1,\Sigma_{-t}^X\eta\wedge(\psi_{-t}^X)^*\phi_2\rangle {\rm dvol}_M\nonumber\\
+\int_M \langle d^\eta\phi_1,\widetilde\Sigma_{-t}^X d^\eta\phi_2\rangle {\rm dvol}_M=0.
\end{eqnarray}
In the same way, we apply \eqref{eq:integralsigmat} to $d^\eta(*\phi_1)$ and $d^\eta(*\phi_2)$ and use \eqref{eq:differencestar} to get
\begin{eqnarray}\label{eq:deltasigma}
\int_M \langle d^\eta(*\Sigma_t^X\phi_1),d^\eta(*\phi_2)\rangle {\rm dvol}_M+\int_M \langle d^\eta(*\widetilde\Sigma_t^X\phi_1),d^\eta(*\phi_2)\rangle {\rm dvol}_M\nonumber\\+i\int_M \langle \Sigma_t^X\eta\wedge(\psi_t^X)^*(*\phi_1),d^\eta(*\phi_2)\rangle {\rm dvol}_M+\int_M \langle d^\eta(*\phi_1),d^\eta(*\Sigma_{-t}^X\phi_2)\rangle {\rm dvol}_M\nonumber\\
+\int_M \langle d^\eta(*\phi_1),d^\eta(*\widetilde\Sigma_{-t}^X\phi_2)\rangle {\rm dvol}_M-i\int_M\langle d^\eta(*\phi_1),\Sigma_{-t}^X\eta\wedge(\psi_{-t}^X)^*(*\phi_2)\rangle {\rm dvol}_M\nonumber\\
+\int_M \langle d^\eta(*\phi_1),\widetilde\Sigma_{-t}^X d^\eta(*\phi_2)\rangle {\rm dvol}_M=0.\nonumber\\
\end{eqnarray}
Adding \eqref{eq:detasigma} and \eqref{eq:deltasigma}, and using $d^\eta*=\pm *\delta^\eta$ and the definition of the magnetic Dirichlet integral, we get
\begin{eqnarray*}
\mathcal{D}^\eta(\Sigma_t^X\phi_1,\phi_2)+\mathcal{D}^\eta(\phi_1,\Sigma_{-t}^X\phi_2)=\\
-i\int_M \langle \Sigma_t^X\eta\wedge(\psi_t^X)^*\phi_1,d^\eta\phi_2\rangle {\rm dvol}_M+i\int_M\langle d^\eta\phi_1,\Sigma_{-t}^X\eta\wedge(\psi_{-t}^X)^*\phi_2\rangle {\rm dvol}_M\nonumber\\
-\int_M \langle d^\eta\phi_1,\widetilde\Sigma_{-t}^X d^\eta\phi_2\rangle {\rm dvol}_M
-\int_M \langle d^\eta(*\widetilde\Sigma_t^X\phi_1),d^\eta(*\phi_2)\rangle {\rm dvol}_M\nonumber\\-i\int_M \langle \Sigma_t^X\eta\wedge(\psi_t^X)^*(*\phi_1),d^\eta(*\phi_2)\rangle {\rm dvol}_M
-\int_M \langle d^\eta(*\phi_1),d^\eta(*\widetilde\Sigma_{-t}^X\phi_2)\rangle {\rm dvol}_M\\+i\int_M\langle d^\eta(*\phi_1),\Sigma_{-t}^X\eta\wedge(\psi_{-t}^X)^*(*\phi_2)\rangle {\rm dvol}_M
-\int_M \langle d^\eta(*\phi_1),\widetilde\Sigma_{-t}^X d^\eta(*\phi_2)\rangle {\rm dvol}_M.
\end{eqnarray*}
By the Cauchy-Schwarz inequality and \cite[Lem. 1.3.9]{Sch:95}, we can bound the right-hand side by
\begin{eqnarray*}
\left(\tilde{C}_1||\Sigma_t^X\eta||_{L^2}||\phi_2||_{H^1}+\tilde{C}_2||\Sigma_{-t}^X\eta||_{L^2}||\phi_2||_{H^1}+\tilde{C}_3||\widetilde\Sigma_{-t}^Xd^\eta\phi_2||_{L^2}\right)||\phi_1||_{H^1}\\
+ (\tilde{C}_4||\widetilde\Sigma_t^X\phi_1||_{H^1}+\tilde{C}_5||\Sigma_t^X\eta||_{L^2}||\phi_1||_{H^1})||\phi_2||_{H^1}\\
+\Big(\tilde{C}_6||\widetilde\Sigma_{-t}^X\phi_2||_{H^1}+\tilde{C}_7||\Sigma_{-t}^X\eta||_{L^2}||\phi_2||_{H^1}+\tilde{C}_8||\widetilde\Sigma_{-t}^Xd^\eta(*\phi_2)||_{L^2}\Big)||\phi_1||_{H^1}.
\end{eqnarray*}
Since $||\widetilde\Sigma_t^X\cdot||_{H^s}\leq c||\cdot||_{H^s}$ for any $s$ and any $t$ (and a constant $c$ independent of $s, t$),  and we have Inequality \eqref{eq:sigmal2} for $\eta$, we deduce that all the above terms can be bounded by $\tilde{C}_9 ||\phi_1||_{H^1}||\phi_2||_{H^1}$, where $\tilde{C}_9$ is a constant depending on $\eta$ (but not on $t$).  Therefore, we get that
\begin{eqnarray*}
|\mathcal{D}^\eta(\Sigma_t^X\phi_1,\phi_2)|-|\mathcal{D}^\eta(\phi_1,\Sigma_{-t}^X\phi_2)|&\leq &|\mathcal{D}^\eta(\Sigma_t^X\phi_1,\phi_2)+\mathcal{D}^\eta(\phi_1,\Sigma_{-t}^X\phi_2)|\\&\leq & \tilde{C}_9 ||\phi_1||_{H^1}||\phi_2||_{H^1}.
\end{eqnarray*}

Finally, we prove Part \ref{part3}. First, it is not difficult to check that if $\hat\xi\in H^1(\Omega_{\rm rel}^k(M,\mathbb{C}))$, then $\Sigma_t^X\hat\xi$ is also in $H^1(\Omega_{\rm rel}^k(M,\mathbb{C}))$ for any $t\neq 0$. Now, we apply Inequality \eqref{eq:detesigma} to $\phi_1=\phi_D$  and to $\phi_2=\Sigma_t^X\phi_D$, which are both in $H^1(\Omega_{\rm rel}^k(M,\mathbb{C}))$, to get 
\begin{eqnarray*}
|\mathcal{D}^\eta(\Sigma_t^X\phi_D,\Sigma_t^X\phi_D)|&\leq& |\mathcal{D}^\eta(\phi_D,\Sigma_{-t}^X\Sigma_t^X\phi_D)|+C_2||\phi_D||_{H^1}||\Sigma_t^X\phi_D||_{H^1}\\
&\stackrel{\eqref{eq:weakboundaryproblem}}{=}&|(\theta,\Sigma_{-t}^X\Sigma_t^X\phi_D)_{L^2}|+C_2||\phi_D||_{H^1}||\Sigma_t^X\phi_D||_{H^1}\\
&\stackrel{\eqref{eq:sigmal2}}{\leq} &\left(C_1||\theta||_{L^2}+C_2||\phi_D||_{H^1}\right)||\Sigma_t^X\phi_D||_{H^1},
\end{eqnarray*}
where $C_1,C_2$ are independent of $t$. Using the magnetic Gaffney inequality (Lemma \ref{lem:gaffny}), we deduce that 
\begin{eqnarray*}
||\Sigma_t^X\phi_D||_{H^1}^2&\leq& C\left(\mathcal{D}^\eta(\Sigma_t^X\phi_D,\Sigma_t^X\phi_D)+||\Sigma_t^X\phi_D||_{L^2}^2\right)\\
&\leq & C\left(C_1||\theta||_{L^2}+C_2||\phi_D||_{H^1}+||\Sigma_t^X\phi_D||_{L^2}\right)||\Sigma_t^X\phi_D||_{H^1},
\end{eqnarray*}
which, after simplifying and using \eqref{eq:sigmal2} once more, gives that 
$$||\Sigma_t^X\phi_D||_{H^1}\leq C(C_1||\theta||_{L^2}+C_2||\phi_D||_{H^1}+C_1||\phi_D||_{H^1})=\hat C.$$
This finishes the proof.
\end{proof}

Next, we prove the following proposition. 

\begin{proposition}[{c.f. \cite[Lem 2.3.2 and 2.3.3]{Sch:95} for the non-magnetic versions}] \label{prop:strongsolution}  Let $(M^m,g)$ be a compact manifold of dimension $m$ with smooth boundary. If $\phi_D\in H^1(\Omega^k_{\rm rel}(M,\mathbb{C}))$ is a Dirichlet potential associated with some $\theta\in L^2$, then it is in $H^2(\Omega^k(M, \mathbb{C}))$. Therefore, it is a solution of the boundary problem 
\begin{equation}\label{eq:boundaryproblemmagnetic}
\left\{
\begin{matrix}
	\Delta^\eta\omega=\theta & \text{on $M$,}\\\\
	\iota^*\omega=0,\,\, \iota^*(\delta^\eta\omega)=0 & \text{on $\partial M$.}
\end{matrix}\right.
\end{equation}  
\end{proposition}
\begin{proof}
We will prove that for any point $p\in M$, there exists a neighborhood $U$ of $p$ in $M$ such that $\phi_D|_U \in H^2(\Omega^k(U, \mathbb{C}))$. 

First, we prove it for a point $p\in M\setminus\partial M$. Let $U$ be a neighborhood of such a $p$ with $U\subset V$ where $V$ is compact and $V\cap \partial M=\emptyset$. Let $f\in C^\infty(M,\mathbb{R})$ be any cut-off function such that $f=1$ on $U$ and ${\rm supp}(f)\subset V$. Let $(e_1,\ldots, e_m)$ be a local orthonormal frame of $V$ and take the vector fields $X_j:=fe_j$ on $M$. Clearly the vector fields $X_j$ satisfy the assumptions of Lemma \ref{eq:estimateh1dirichlet} and, by Part \ref{part3}, we have $||\Sigma_t^{X_j}\phi_D||_{H^1}\leq \hat{C}$, which means that, for any $t_l\in (-1,1)$ with $t_l\to 0$,  the sequence $(\Sigma_{t_l}^{X_j}\phi_D)_{l}$ is bounded in $H^1$. Hence by \cite[Cor. 1.5.2]{Sch:95}, it has an $L^2$-convergent subsequence that we still denote it by $(\Sigma_{t_l}^{X_j}\phi_D)_{l}$, with limit in $H^1$. 
For any $\omega \in H^1$, $\lim_{t \to 0} \Sigma_t^X \omega = \mathcal{L}_X \omega$ (see \cite[Sec. 2.3]{Sch:95}).
So, as $\phi_D\in H^1$, we have that $||\Sigma_{t_l}^{X_j}\phi_D-\mathcal{L}_{X_j}\phi_D||_{L^2}\to 0$ as $l\to \infty$. Hence $\mathcal{L}_{X_j}\phi_D\in H^1$ for any $j$. Hence $\mathcal{L}_{e_j}\phi_D\in H^1(\Omega^k(U, \mathbb{C}))$ and, then, by the Meyers-Serrin theorem \cite[Thm. 1.3.3]{Sch:95}, $\phi_D\in H^2(\Omega^k(U, \mathbb{C}))$. This finishes the proof for a point $p$ in the interior of $M$. 

Next, we fix a point $p\in \partial M$ and take a neighbourhood $U$ of $p$ in $M$ and an orthonormal frame $\{\tilde\nu, e_2,\ldots,e_m\}$ such that $\tilde\nu|_{\partial M}=\nu$, the normal vector field to the boundary, and $e_j|_{\partial M}\in T\partial M$, for $j\geq 2$. We do as before by taking the vector fields $X_j=fe_j$ and a similar argument gives that $\mathcal{L}_{e_j}\phi_D$ (or equivalently $\nabla_{e_j}\phi_D$) is an element in  $H^1(\Omega^k(U, \mathbb{C}))$ for all $j\geq 2$. Hence, it is sufficient to study the differentiability of $\nabla_{\tilde\nu}\phi_D$. First, if we denote by $\Omega_0^k(M)$ the space of smooth forms which have compact support  in the interior $M\setminus\partial M$, then by Green's formula and the fact that $\phi_D\in H^2$ in the interior of $M$, Equation \eqref{eq:weakboundaryproblem} becomes 
$$(\Delta^\eta\phi_D,\hat\xi)_{L^2}=(\theta,\hat\xi)_{L^2},$$
for any $\hat\xi\in \Omega_0^k(M)$. Now, the density of $\Omega_0^k(M)$ in $L^2$ gives that $\Delta^\eta\phi_D=\theta$ in $L^2$. Using the magnetic Weitzenb\"ock formula \cite[Eq. 3.10]{EGHP:23}, we write 
\begin{eqnarray}\label{eq:bochnermagneticphid}
\Delta^\eta\phi_D&=&-\nabla_{\tilde\nu}\nabla_{\tilde\nu}\phi_D+\nabla_{\nabla_{\tilde\nu}\tilde\nu}\phi_D-\sum_{j\geq 2}(\nabla_{e_j}\nabla_{e_j}\phi_D-\nabla_{\nabla_{e_j}e_j}\phi_D)+\mathcal{B}^{[p],\eta}\phi_D\nonumber\\
&&+i(\delta^M\eta) \phi_D-2i\nabla_{\eta}\phi_D+|\eta|^2\phi_D. 
\end{eqnarray}
It is now clear that 
$$
\nabla_{\nabla_{\tilde\nu}\tilde\nu}\phi_D-\sum_{j\geq 2}(\nabla_{e_j}\nabla_{e_j}\phi_D-\nabla_{\nabla_{e_j}e_j}\phi_D)+\mathcal{B}^{[p],\eta}\phi_D+i(\delta^M\eta) \phi_D-2i\nabla_{\eta}\phi_D+|\eta|^2\phi_D$$
is an element in $L^2(\Omega^k(U, \mathbb{C}))$, since $\phi_D\in H^1(\Omega^k(U, \mathbb{C}))$ and $\nabla_{e_j}\phi_D\in H^1(\Omega^k(U,\mathbb{C}))$ as we have previously shown. Moreover, since $\Delta^\eta\phi_D=\theta\in L^2(\Omega^k(U,\mathbb{C}))$, we deduce then from \eqref{eq:bochnermagneticphid} that $\nabla_{\tilde\nu}\nabla_{\tilde\nu}\phi_D\in L^2(\Omega^k(U,\mathbb{C}))$. On the other hand, as the curvature operator $R(X,Y)=[\nabla_X,\nabla_Y]-\nabla_{[X,Y]}$ is a smooth endomorphism on $\Omega^k(U, \mathbb{C})$, we have $R(X,Y)\phi_D\in H^1(\Omega^k(U,\mathbb{C}))$. Hence, by writing
\begin{equation*}
\nabla_{e_j}\nabla_{\tilde\nu}\omega=\nabla_{\tilde\nu}\nabla_{e_j}\omega+\nabla_{[e_j,\tilde\nu]}\omega+R(e_j,\tilde\nu)\omega,
\end{equation*}
we get that $\nabla_{e_j}\nabla_{\tilde\nu}\phi_D\in L^2(\Omega^k(U,\mathbb{C}))$. The result then follows by the Meyers-Serrin theorem. 

To prove the last part, we have already shown that $\Delta\phi_D=\theta\in L^2$. Now, using Green's formula, for any $\hat\xi \in H^1(\Omega^k_{\rm rel}(M,\mathbb{C}))$ we have that
\begin{eqnarray*}
\mathcal{D}^\eta(\phi_D,\hat\xi)&=&\int_M\langle\Delta^\eta\phi_D,\hat\xi\rangle {\rm dvol}_M-\int_{\partial M}\langle\nu\lrcorner d^\eta\phi_D,\iota^*\hat\xi\rangle {\rm dvol}_M+\int_{\partial M}\langle\iota^* \delta^\eta\phi_D,\nu\lrcorner\hat\xi\rangle {\rm dvol}_M\\
&=&\int_M\langle\theta,\hat\xi\rangle {\rm dvol}_M+\int_{\partial M}\langle\iota^* \delta^\eta\phi_D,\nu\lrcorner\hat\xi\rangle {\rm dvol}_{\partial M}.
\end{eqnarray*}
Hence, Equation \eqref{eq:weakboundaryproblem} gives that $\displaystyle\int_{\partial M}\langle\iota^* \delta^\eta\phi_D,\nu\lrcorner\hat\xi\rangle {\rm dvol}_M=0$ for any $\hat\xi \in H^1(\Omega^k_{\rm rel}(M,\mathbb{C}))$. By choosing $\hat\xi$ of the form $f\hat\nu\wedge \delta^\eta\phi_D\in H^1(\Omega_{\rm rel}(M,\mathbb{C}))$ where $\hat\nu$ is a smooth local extension of $\nu$ in some local neighborhood of the boundary and $f$ is a cut-off function equal to $1$ on $\partial M$, we get that $\iota^* \delta^\eta\phi_D=\nu\lrcorner\hat\xi$ and, thus, $\iota^* \delta^\eta\phi_D=0$ on $\partial M$. This finishes the proof. 
\end{proof}

Now, we come to the proof of Theorem \ref{thm:smoothcohomology}.

\begin{proof} [Proof of Theorem \ref{thm:smoothcohomology}] Let us consider the boundary problem \eqref{eq:boundaryproblemmagnetic}. It is not so difficult to check that this boundary problem is elliptic in the sense of Lopatinskii-Shapiro, since the principal symbol of $\Delta^\eta$ is the same as the one of the Hodge Laplacian $\Delta$ without magnetic potential and the principal symbol of $\delta^\eta$ is the same as the one of $\delta$. Hence by \cite[Lem. 1.6.5]{Sch:95}, it is elliptic. Therefore, by \cite[Thm. 1.6.2]{Sch:95}, for  any smooth $\theta$, any solution of \eqref{eq:boundaryproblemmagnetic} in $H^s$ with $s\geq 2$ is smooth. 

Consider now an element $\omega$ in $\mathcal{H}^{k,\eta}_{\rm rel}(M)$, then $\omega$ is a Dirichlet potential associated with $\theta=0$, as $\mathcal D^\eta(\omega,\hat\xi)=0$ for any $\hat\xi\in H^1(\Omega^k_{\rm rel}(M,\mathbb{C}))$. Therefore, by Proposition \ref{prop:strongsolution}, it is an element in $H^2$. Thus $\omega$ is smooth. This finishes the proof.  
\end{proof}

\subsection{Proof of Theorem \ref{thm:magneticboundprob}} 
We now combine the results of Sections \ref{ss:unique} and \ref{ss:exist} to prove Theorem \ref{thm:magneticboundprob} from which we have well-posedness of \eqref{eq:magneticsteklovsolut} and hence of the magnetic Steklov problem \eqref{eq:steklovmagnetic}.

\begin{proof}[Proof of Theorem \ref{thm:magneticboundprob}]
We denote by $\Omega_0^k(M,\mathbb{C})$ the space of complex differential forms that vanish on the boundary and by $\mathcal{H}^{k,\eta}(M)=\{\omega\in H^1|\,\, d^\eta\omega=0,\delta^\eta\omega=0\}$. We know from Theorem \ref{thm:dirichletmagnetic} that $\Omega_0^k(M,\mathbb{C})\cap \mathcal{H}^{k,\eta}(M)=0$.  

Next, we show that $\mathcal{D}^\eta$ is $H^1$-elliptic on $H^1(\Omega_0^k(M,\mathbb{C}))$, i.e. that
\begin{equation}\label{eq:ellipticsteklov}
c||\omega||^2_{H^1}\leq \mathcal{D}^\eta(\omega,\omega)\leq C||\omega||^2_{H^1},
\end{equation}
for any $\omega\in H^1(\Omega_0^k(M,\mathbb{C}))$. For this, we follow the same steps as in the proof of Proposition \ref{lem:dirichletpotential}. First, notice that the right-hand side of Inequality \eqref{eq:ellipticsteklov} is a consequence of the continuity of $\mathcal{D}^\eta$. To prove the left-hand side, we consider the unit $L^2$-sphere in $H^1(\Omega_0^k(M,\mathbb{C}))$ and we construct a sequence $(\hat\omega_j)_j$ in $S_1$ such that $\mathcal{D}^\eta(\hat\omega_j,\hat\omega_j)\to \mathop{\rm inf}\limits_{\hat\omega\in S_1}(\mathcal{D}^\eta(\hat\omega,\hat\omega))$. By the Rellich lemma (see, e.g., \cite[Thm 1.3.3(c)]{Sch:95}), $H^1(\Omega_0^k(M,\mathbb{C}))\hookrightarrow L^2(\Omega^k(M,\mathbb{C}))$ is compact, so the sequence $(\hat\omega_j)_j$ has a convergent subsequence in $L^2$ with limit $\hat\theta$ in $H^1(\Omega_0^k(M,\mathbb{C}))$. The form $\hat\theta$ has $L^2$-norm equal to $1$ and $\mathcal{D}^\eta(\hat\theta,\hat\theta)>0$. The latter property comes from the fact that, if $\mathcal{D}^\eta(\hat\theta,\hat\theta)=0$, then one would get $d^\eta\hat\theta=\delta^\eta\hat\theta=0$ on $M$ and $\hat\theta=0$ on $\partial M$. Hence, $\hat\theta$ would be an element in the magnetic relative cohomology $\mathcal{H}^{k,\eta}_{\rm rel}(M)$ and, therefore, it should be smooth by Theorem \ref{thm:smoothcohomology} so must be zero, which is a contradiction to $||\hat\theta||_{L^2}=1$. The rest of the proof is completely analogous to the proof given in Proposition \ref{lem:dirichletpotential}. Hence Inequality \eqref{eq:ellipticsteklov} is established. 

As the form $\mathcal{D}^\eta$ is $H^1$-elliptic on $H^1(\Omega_0^k(M,\mathbb{C}))$, we apply the Lax-Milgram theorem for given $\theta\in \Omega^k(M,\mathbb{C})$, to deduce the existence of $\phi_0\in H^1(\Omega_0^k(M,\mathbb{C}))$ such that 
$$\mathcal{D}^\eta(\phi_0,\hat\xi)=(\theta,\hat\xi)_{L^2}$$
for any $\hat\xi\in H^1(\Omega_0^k(M,\mathbb{C}))$. Now, since $\phi_0$ vanishes on the boundary, we deduce from Proposition  \ref{prop:strongsolution} that it is in $ H^2(\Omega_0^k(M,\mathbb{C}))$ and, for any $\theta\in \Omega^k(M,\mathbb{C})$, is a solution of the problem 
\begin{equation}\label{eq:steklohomog}
\Delta^\eta\phi_0=\theta\quad\text{and}\quad \phi_0|_{\partial M}=0.
\end{equation}

Finally, we come back to Problem \eqref{eq:magneticsteklovsolut} and show the existence of solutions. Let $\varphi\in \Omega^k(M, \mathbb{C})$ and $\psi\in \Omega^k(M, \mathbb{C})|_{\partial M}$. We consider $\tilde\psi$ any extension of $\psi$ to all of $M$ and set $\theta:=\Delta^\eta\tilde\psi-\varphi$. We denote by $\phi_0$ the solution of \eqref{eq:steklohomog} for such $\theta$, and set $\omega:=\tilde\psi-\phi_0$. Then $\omega$ satisfies \eqref{eq:magneticsteklovsolut}.
\end{proof}

\section{Obstruction to the Diamagnetic Inequality}\label{ss:diamagnetic}
The main goal of this section is to prove that an analogue of the Diamagnetic Inequality does not hold in general for the magnetic Steklov operator on differential forms.

\subsection{Spectral properties of the magnetic Steklov operator on differential forms}
By following \cite{RS:12}, we present some spectral properties of this operator.

\begin{proposition}
    Let $(M^m,g)$ be a compact Riemannian manifold of dimension $m$ with smooth boundary and let $\eta$ be a magnetic potential. Then 
    \begin{enumerate}
        \item[(a)] The operator $T^{[k],\eta}$  is non-negative and self-adjoint. 
        \item[(b)] The kernel of $T^{[k],\eta}$ is the magnetic absolute cohomology. 
    \end{enumerate}
\end{proposition}

\begin{proof} The proof follows the same steps as in \cite[Thm. 11]{RS:12}. 
For part (a), let $\omega_1,\omega_2\in \Omega^k(\partial M,\mathbb{C})$ and denote by $\hat\omega_1,\hat\omega_2$ their $\eta$-harmonic extensions with respect to Problem \eqref{eq:steklovmagnetic}. By the magnetic Stokes formula, 
$$\int_M\langle d^\eta\hat\omega_1,\hat\omega_2\rangle {\rm dvol}_M=\int_M\langle \hat\omega_1,\delta^\eta\hat\omega_2\rangle {\rm dvol}_M-\int_{\partial M}\langle \iota^*\hat\omega_1,\nu\lrcorner\hat\omega_2\rangle {\rm dvol}_M.$$
 So we have that
 \begin{eqnarray*}
 \int_{\partial M} \langle T^{[k],\eta}\omega_1,\omega_2\rangle {\rm dvol}_M&=&-\int_{\partial M} \langle \nu\lrcorner d^\eta\hat\omega_1,\hat\omega_2\rangle {\rm dvol}_M\\
 &=&-\int_{\partial M} \overline{\langle \iota^*\hat\omega_2,\nu\lrcorner d^\eta\hat\omega_1\rangle} {\rm dvol}_M\\
 &=&\int_M\langle d^\eta\hat\omega_1,d^\eta\hat\omega_2\rangle {\rm dvol}_M-\int_M\langle\delta^\eta d^\eta\hat\omega_1,\hat\omega_2\rangle {\rm dvol}_M\\
 &=&\int_M\langle d^\eta\hat\omega_1,d^\eta\hat\omega_2\rangle {\rm dvol}_M+\int_M\langle d^\eta\delta^\eta \hat\omega_1,\hat\omega_2\rangle {\rm dvol}_M\\
 &=&\int_M\langle d^\eta\hat\omega_1,d^\eta\hat\omega_2\rangle {\rm dvol}_M+\int_M\langle \delta^\eta \hat\omega_1,\delta^\eta\hat\omega_2\rangle {\rm dvol}_M,
 \end{eqnarray*}
where we use that $\hat\omega_1$ is $\eta$-harmonic in the fourth equality, and that $\nu \lrcorner\hat\omega_2 = 0$ on $\partial M$ in the fifth equality. This shows that $T^{[k],\eta}$ is self-adjoint and nonnegative. Moreover, its kernel is given by the magnetic absolute cohomology. 
\end{proof}

The following theorem asserts that the magnetic Steklov operator has discrete spectrum.

\begin{theorem}\label{thm:discretespectrum}
Let $(M^m,g)$ be a compact Riemannian manifold of dimension $m$ with smooth boundary and let $\eta$ be a magnetic potential.
Then $T^{[k],\eta}$ has a discrete non-negative spectrum,
$$\sigma_{1,k}^\eta(M)\leq \sigma_{2,k}^\eta(M)\leq \ldots \nearrow +\infty,$$
with finite multiplicities, and the magnetic Steklov eigen $k$-forms restricted to the boundary form an orthonormal basis of $L^2(\Omega^k(\pam, \CC))$.
\end{theorem}

To prove Theorem \ref{thm:discretespectrum}, we adapt the work of Arendt and Mazzeo \cite{AM:07} as follows.

\begin{definition}
Let $V$ be a Hilbert space that is densely and continuously embedded in a Hilbert space $H$. A symmetric, continuous, bilinear form $\mathcal{B}:V\times V\mapsto \RR$ is said to be $(V,H)$-elliptic if there exist $\beta\in \RR$ and $\alpha>0$ such that
\[
\mathcal{B}(x,x)+\beta\norm{x}{H}^2\geq \alpha\norm{x}{V}^2\quad \forall \, x\in V.
\]    
\end{definition}
\noindent \textbf{Note. }To each such form is associated a self-adjoint operator, which further has a compact resolvent if and only if the inclusion $V \hookrightarrow H$ is compact.

We use the following weak formulations:
\begin{itemize}
\item For $\omega\in H^1(\Omega^k_{\rm abs}(M, \mathbb{C}))$, we say that $\Delta^\eta \omega =0$ if
\[
\int_M \langle d^\eta \omega, d^\eta \phi \rangle \dvm + \int_M \langle \delta^\eta \omega, \delta^\eta \phi \rangle \dvm = 0, \qquad \forall\, \phi\in H^1(\Omega^k_0(M, \mathbb{C})).
\]
\item For $\omega\in H^1(\Omega^k_{\rm abs}(M, \mathbb{C}))$ such that $\Delta^\eta \omega =0$, we say that $\nu\lrcorner d^\eta \omega \in L^2(\Omega^k(\pam,\mathbb{C}))$ if there exists $\beta\in L^2(\Omega^k(\pam,\mathbb{C}))$ such that
\[
\int_M \langle d^\eta \omega, d^\eta \phi \rangle \dvm + \int_M \langle \delta^\eta \omega, \delta^\eta \phi \rangle \dvm = \int_\pam \langle \beta, \iota^* \phi \rangle \dvpm, \qquad \forall \, \phi \in H^1(\Omega^k_{\rm abs}(M, \mathbb{C})).
\]
We may then set $-\nu\lrcorner d^\eta \omega=\beta$. 
\end{itemize}

\begin{lemma}[Ehrling's inequality; see Lemma 2.3 of \cite{AM:07}] \label{lem: Ehrling ineq}
Let $\mathcal{B}_1$, $\mathcal{B}_2$, $\mathcal{B}_3$ be Banach spaces, with $\mathcal{B}_1$ reflexive. Let $T:\mathcal{B}_1\to \mathcal{B}_3$ be a compact map, and $S:\mathcal{B}_1\to \mathcal{B}_2$ a continuous injection. Then, for any $\epsilon>0$, there exists $C_\epsilon>0$ such that
\end{lemma}
\[
\norm{Tx}{\mathcal{B}_3}\leq \epsilon\norm{x}{\mathcal{B}_1}+C_\epsilon\norm{Sx}{\mathcal{B}_2},\qquad \forall x\in \mathcal{B}_1.
\]
Define,
\[
\mathrm{Dom}(T^{[k],\eta}):= \{\omega \in L^2(\Omega^k(\pam,\mathbb{C})): \text{$\exists$ $\tilde\omega \in H^1(\Omega^k_{\rm abs}(M, \mathbb{C}))$ such that  $\iota^*\tilde \omega=\omega$, $\Delta^\eta \tilde\omega=0$ and $\nu\lrcorner d^\eta \tilde\omega \in L^2(\Omega^k(\pam,\mathbb{C}))$}\},
\]
and set
\[
T^{[k],\eta}\omega=-\nu\lrcorner d^\eta \tilde\omega.
\]

\begin{lemma}
We have
\[
H^1(\Omega^k_{\rm abs}(M, \mathbb{C}))=H^1(\Omega^k_0(M,\mathbb{C}))\oplus H^1_{\eta,\rm har}(\Omega^k_{\rm abs}(M, \mathbb{C})),
\]
where $H^1_{\eta,\rm har}(\Omega^k_{\rm abs}(M, \mathbb{C})):= \{\tilde\omega \in H^1(\Omega^k_{\rm abs}(M, \mathbb{C})): \Delta^\eta \tilde \omega=0 \}$. 
%Is it equal to $\mathcal{H}^{k,\eta}_{\rm abs}$, or is it a strictly larger set?
\end{lemma}

The above lemma implies that the trace operator restricted to $H^1_{\eta,\rm har}(\Omega^k_{\rm abs}(M, \mathbb{C}))$ is bijective. Similar to the discussion in \cite[Section 3]{AM:07}, using the closed graph theorem and the Stone--Weierstrass theorem, we would have that the space $$V:=\iota^*H^1(\Omega^k_{\rm abs}(M, \mathbb{C}))$$ with the well-defined norm $\norm{\iota^*\omega}{V}:=\norm{\omega}{H^1(\Omega^k_{\rm abs}(M, \mathbb{C}))}$ becomes a Hilbert space, and is dense in $L^2(\Omega^k(\pam,\mathbb{C}))$.

\begin{proposition}\label{prop:assocform}
The operator $T^{[k],\eta}$ is associated with the symmetric, continuous and $(V, L^2(\Omega^k(\pam,\mathbb{C})))$-elliptic bilinear form $\mathcal{A}^\eta$ on $V$ given by
\[
\mathcal{A}^\eta(\alpha, \beta):=\int_M \langle d^\eta\tilde\alpha,d^\eta\tilde\beta \rangle\dvm+\int_M \langle \delta^\eta\tilde\alpha,\delta^\eta\tilde\beta \rangle\dvm,
\]
where $\tilde \alpha,\tilde \beta \in H^1_{\eta,\rm har}(\Omega^k_{\rm abs}(M, \mathbb{C}))$ such that $\iota^*\tilde \alpha=\alpha$ and $\iota^*\tilde\beta=\beta$.
\end{proposition}
Due to the compactness of the trace operator $H^1_{\eta,\rm har}(\Omega^k_{\rm abs}(M, \mathbb{C}))\to L^2(\Omega^k(\pam,\mathbb{C}))$, the above proposition implies that $T^{[k],\eta}$ has a compact resolvent, whence it follows that it has a discrete spectrum diverging to infinity, with finite multiplicities, and the eigen $k$-forms restricted to the boundary form an orthonormal basis of $L^2(\Omega^k(\pam, \CC))$. We now prove the proposition.
\begin{proof}[Proof of Proposition \ref{prop:assocform}]
We use Lemma \ref{lem: Ehrling ineq} with the compact embedding $H^1_{\eta,\rm har}(\Omega^k_{\rm abs}(M, \mathbb{C}))\hookrightarrow L^2(\Omega^k(M, \mathbb{C}))$ and the continuous injection given by the trace operator $H^1_{\eta,\rm har}(\Omega^k_{\rm abs}(M, \mathbb{C}))\to L^2(\Omega^k(\pam,\mathbb{C}))$. Let $C_G^\eta>0$ be the constant from the magnetic Gaffney inequality that depends on the manifold $M$ and the magnetic potential $\eta$. Note that we can use the absolute conditions in place of the relative conditions, although we derived the Magnetic Gaffney inequality for the latter. Given $0<a<(C_G^\eta)^{-1}$, there exists $c>0$ such that for all $\omega \in H^1_{\eta,\rm har}(\Omega^k_{\rm abs}(M, \mathbb{C}))$,
\begin{align*}
\norm{\omega}{L^2(\Omega^k(M, \mathbb{C}))}^2&\leq a \,\norm{\omega}{H^1(\Omega^k(M, \mathbb{C}))}+c\,\norm{\iota^*\omega}{L^2(\Omega^k(\pam,\mathbb{C}))}\\
&\leq a \, C_G^\eta \left(\norm{d^\eta \omega}{L^2(\Omega^k(M,\mathbb{C}))}^2+\norm{\delta^\eta\omega}{L^2(\Omega^k(M,\mathbb{C}))}^2+\norm{\omega}{L^2(\Omega^k(M,\mathbb{C}))}^2\right)+c\,\norm{\iota^*\omega}{L^2(\Omega^k(\pam,\mathbb{C}))}^2,
\end{align*}
which gives
\[
\norm{\omega}{L^2(\Omega^k(M,\mathbb{C}))}^2\leq \frac{a\,C_G^\eta}{1-a\,C_G^\eta}\left(\norm{d^\eta \omega}{L^2(\Omega^k(M,\mathbb{C}))}^2+\norm{\delta^\eta\omega}{L^2(\Omega^k(M,\mathbb{C}))}^2\right)+ \frac{c}{1-a\,C_G^\eta}\norm{\iota^*\omega}{L^2(\Omega^k(\pam,\mathbb{C}))}^2.
\]
Without loss of generality, we may assume that $C_G^\eta>1$. Choosing $a=\frac{C_G^\eta - 1}{(C_G^\eta)^2}<(C_G^\eta)^{-1}$, we have that
\[
\norm{\omega}{L^2(\Omega^k(M,\mathbb{C}))}^2\leq (C_G^\eta-1)\left(\norm{d^\eta \omega}{L^2(\Omega^k(M,\mathbb{C}))}^2+\norm{\delta^\eta\omega}{L^2(\Omega^k(M,\mathbb{C}))}^2\right)+ c\, C_G^\eta\norm{\iota^*\omega}{L^2(\Omega^k(\pam,\mathbb{C}))}^2, \quad \forall \, \omega \in H^1_{\eta,\rm har}(\Omega^k_{\rm abs}(M, \mathbb{C})).
\]
We observe that
\begin{align*}
\mathcal{A}^\eta(\iota^*\omega, \iota^*\omega)&=\norm{d^\eta \omega}{L^2(\Omega^k(M,\mathbb{C}))}^2+\norm{\delta^\eta\omega}{L^2(\Omega^k(M,\mathbb{C}))}^2+(C_G^\eta)^{-1} \norm{\omega}{L^2(\Omega^k(M,\mathbb{C}))}^2-(C_G^\eta)^{-1} \norm{\omega}{L^2(\Omega^k(M,\mathbb{C}))}^2\\
&\begin{multlined}
\geq \norm{d^\eta \omega}{L^2(\Omega^k(M,\mathbb{C}))}^2+\norm{\delta^\eta\omega}{L^2(\Omega^k(M,\mathbb{C}))}^2+(C_G^\eta)^{-1} \norm{\omega}{L^2(\Omega^k(M,\mathbb{C}))}^2\\-(1-(C_G^\eta)^{-1})\left(\norm{d^\eta \omega}{L^2(\Omega^k(M,\mathbb{C}))}^2+\norm{\delta^\eta\omega}{L^2(\Omega^k(M,\mathbb{C}))}^2\right)-c\norm{\iota^*\omega}{L^2(\Omega^k(\pam,\mathbb{C}))}^2,
\end{multlined}
\end{align*}
implying
\begin{align*}
\mathcal{A}^\eta(\iota^*\omega, \iota^*\omega)+c\norm{\iota^*\omega}{L^2(\Omega^k(\pam,\mathbb{C}))}^2&\geq (C_G^\eta)^{-1}\left(\norm{d^\eta \omega}{L^2(\Omega^k(M,\mathbb{C}))}^2+\norm{\delta^\eta\omega}{L^2(\Omega^k(M,\mathbb{C}))}^2+\norm{\omega}{L^2(\Omega^k(M,\mathbb{C}))}^2\right)\\
&\geq (C_G^\eta)^{-2}\norm{\omega}{H^1(\Omega^k(M,\mathbb{C}))}^2\\
&= (C_G^\eta)^{-2}\norm{\iota^*\omega}{V}^2.
\end{align*}
Thus, $\mathcal{A}^\eta$ is $(V, L^2(\Omega^k(\pam,\mathbb{C})))$-elliptic, and is associated with a self-adjoint operator, say $\mathcal{O}$. So, it remains to prove that $\mathcal{O}=T^{[k], \eta}$.

Suppose $\iota^*\omega \in {\rm Dom}(\mathcal{O})$ for $\omega \in H^1_{\eta,\rm har}(\Omega^k_{\rm abs}(M, \mathbb{C}))$,  which is equivalent to 
\[
\int_M \langle d^\eta \omega, d^\eta \phi \rangle \dvm + \int_M \langle \delta^\eta \omega, \delta^\eta \phi \rangle \dvm =\int_\pam \langle \mathcal{O}(\iota^*\omega), \iota^*\phi \rangle \dvpm, \qquad \forall \, \phi \in H^1_{\eta,\rm har}(\Omega^k_{\rm abs}(M, \mathbb{C})).
\]
Observe that the above identity holds  also for all $\phi \in H^1(\Omega^k_0(M))$, as both the left-hand side and the right-hand side would be equal to zero. So, the identity holds for all $\phi \in H^1(\Omega^k_{\rm abs}(M, \mathbb{C}))=H^1(\Omega^k_0(M, \mathbb{C}))\oplus H^1_{\eta,\rm har}(\Omega^k_{\rm abs}(M, \mathbb{C}))$, which implies that $\nu\lrcorner d^\eta \omega \in L^2(\Omega^k(\pam, \mathbb{C}))$ in the weak sense. Hence, $\iota^*\omega \in {\rm Dom}(T^{[k],\eta})$ and $T^{[k],\eta} (\iota^*\omega)=\mathcal{O}(\iota^*\omega)$.

Conversely, suppose $\iota^*\omega \in {\rm Dom}(T^{[k],\eta})$ for $\omega \in H^1_{\eta,\rm har}(\Omega^k_{\rm abs}(M, \mathbb{C}))$. Then
\[
\int_M \langle d^\eta \omega, d^\eta \phi \rangle \dvm + \int_M \langle \delta^\eta \omega, \delta^\eta \phi \rangle \dvm= \int_\pam \langle T^{[k],\eta} \omega, \iota^*\phi \rangle \dvpm,\qquad \forall\, \phi \in H^1(\Omega^k_{\rm abs}(M, \mathbb{C})),
\]
In particular, the identity holds for all $\phi\in H^1_{\eta,\rm har}(\Omega^k_{\rm abs}(M,\mathbb{C}))$, which implies that $\iota^*\omega \in {\rm Dom}(\mathcal{O})$ and $\mathcal{O}(\iota^*\omega)=T^{[k],\eta} (\iota^*\omega)$. Thus, $T^{[k],\eta}$ is the self-adjoint operator associated with $\mathcal{A}^\eta$.
\end{proof}

Finally, we state the min-max principle for the first eigenvalue of the magnetic Steklov operator on forms.
\begin{proposition}\label{prop:minmax}
Let $(M^m,g)$ be a compact Riemannian manifold of dimension $m$ with smooth boundary and let $\eta$ be a magnetic potential. Then the first eigenvalue satisfies the min-max principle:
$$\sigma_{1,k}^\eta(M)=\mathop{\rm inf}\limits_{\{\hat\omega \in \Omega^k(M, \mathbb{C})|\,\nu\lrcorner\hat\omega=0\}}\left\{\frac{\displaystyle\int_M(|d^\eta\hat\omega|^2+|\delta^\eta\hat\omega|^2){\rm dvol}_M}{\displaystyle\int_{\partial M}|\hat\omega|^2{\rm dvol}_{\partial M}}\right\}.$$
\end{proposition}

The proof of Proposition \ref{prop:minmax} can be done in exactly the same way as for the case when $\eta=0$. See the proof of \cite[Thm. 11]{RS:12} for more details.
 
\subsection{Proofs of Theorems \ref{thm:eigtaylor} and \ref{thm:balldiamineqcounter}}

We first prove Theorem \ref{thm:eigtaylor} which shows that the analogue of the Diamagnetic Inequality for the Steklov problem does not necessarily hold.
\begin{proof}[Proof of Theorem \ref{thm:eigtaylor}]
    We follow the same steps as in \cite[Thm. 4.7]{EGHP:23}. We take $\hat\omega$ to be any complex differential $k$-form on $M$ such that $\nu\lrcorner\hat\omega=0$ on $\partial M$. We have $$\int_M|d^{t\eta}\hat\omega|^2 {\rm dvol}_M=||d^M\hat\omega||_{L^2}^2+2t\Re\left(\int_M \langle d^M\hat\omega, i\eta\wedge\hat\omega\rangle {\rm dvol}_M\right)+t^2||\eta\wedge\hat\omega||_{L^2}^2,$$ 
    and
    \begin{eqnarray*}
    \int_M|\delta^{t\eta}\hat\omega|^2 {\rm dvol}_M&=&||\delta^M\hat\omega||_{L^2}^2-2t\Re\left(\int_M \langle \delta^M\hat\omega, i\eta\lrcorner\hat\omega\rangle {\rm dvol}_M\right)+t^2||\eta\lrcorner\hat\omega||_{L^2}^2\\
    &=&||\delta^M\hat\omega||_{L^2}^2-2t\Re\left(\int_M \langle \hat\omega, id^M(\eta\lrcorner\hat\omega)\rangle {\rm dvol}_M\right)+t^2||\eta\lrcorner\hat\omega||_{L^2}^2,
    \end{eqnarray*}
where the last equality is obtained by the Stokes formula as $\nu\lrcorner\hat\omega=0$. 

Now we take $\omega$ to be an eigenform of the Steklov operator $T^{[k]}$ associated to $\sigma_{1,k}(M)$ and $\hat \omega$ to be its harmonic extension. We add the two equations from above, use the Cartan formula $\mathcal{L}_\eta \hat\omega = \eta \lrcorner d^M \hat\omega + d^M(\eta \lrcorner \hat\omega)$, and apply the min-max principle to obtain 
$$\sigma^{t\eta}_{1,k}(M)\leq \sigma_{1,k}(M)+\frac{2t}{||\hat\omega||^2_{L^2(\partial M)}}{\rm Im}\left(\int_M\langle\mathcal{L}_\eta\hat\omega,\hat\omega\rangle {\rm dvol}_M\right)+\frac{||\hat\omega||^2_{L^2}||\eta||_\infty^2}{||\hat\omega||^2_{L^2(\partial M)}}t^2,$$
which is the required inequality.
\end{proof}

Next, we prove Theorem \ref{thm:balldiamineqcounter} which shows that the conditions of Theorem \ref{thm:eigtaylor} are satisfied in a particular case and hence that the Diamagnetic Inequality doesn't always hold.

\begin{proof}[Proof of Theorem \ref{thm:balldiamineqcounter}] 
By Theorem \ref{thm:eigtaylor}, it is sufficient to find a complex eigenform $\omega$ of the Steklov operator on the ball such that ${\rm Im} \left( \int_M \langle \mathcal{L}_{\eta} \hat\omega, \hat\omega \rangle {\rm dvol}_M\right)$ is negative, where $\hat\omega$ is the harmonic extension of $\omega$.  It is shown in \cite[Cor. 4]{RS:14} that the lowest eigenvalue $\sigma_{1,1}(M)$ of the Steklov operator on $1$-forms on $\mathbb{B}^{2n}$ is equal to $\frac{n+1}{n}$. Indeed, they show in \cite[Prop. 7]{RS:14} that any exact $1$-form $\omega=d^{\mathbb{S}^{2n-1}}\phi$, which is an eigenform of the Hodge Laplacian on $\mathbb{S}^{2n-1}$ associated with the eigenvalue $2n-1$, is an eigenform of the Steklov operator associated with the eigenvalue $\frac{n+1}{n}$. Notice here that $\phi$ is an eigenfunction of the Laplace-Beltrami operator on $\mathbb{S}^{2n-1}$ associated with the lowest positive eigenvalue $2n-1$ which is of multiplicity $2n$ \cite[p.159]{BGM:71}. Moreover, they compute in \cite[Prop. 7]{RS:14} the corresponding harmonic extension $\hat\omega$ of $\omega$ and show that it is equal to 
$$\hat\omega=\frac{2n-1}{2n}\left(d\hat\phi+\frac{1}{2n-1}r^3d^{\mathbb{S}^{2n-1}}\phi-r^2\phi dr\right),$$
where $d$ is the exterior differential in $\mathbb{R}^{2n}$ and $\hat\phi$ is the harmonic extension of $\phi$  satisfying $\frac{\partial \hat\phi}{\partial r}=\phi$ on $\mathbb{S}^{2n-1}$. A straightforward computation shows that $\hat\phi=r\phi$ and thus, the harmonic extension of $\omega$ becomes equal to 
\begin{equation}\label{eq:harmonicextensionomega}
\hat\omega=\frac{2n-1}{2n}\left((1-r^2)\phi dr+(r+\frac{1}{2n-1}r^3)d^{\mathbb{S}^{2n-1}}\phi\right).
\end{equation}
Now, for all $j=1,\ldots,n$, we consider the complex functions $\phi_j:=x_j+iy_j$ and  $\overline\phi_j=x_j-iy_j$, which are homogeneous polynomials of degree $1$. These are harmonic functions on $M$ and their restrictions to $\mathbb{S}^{2n-1}$ are eigenfunctions of the Laplace-Beltrami operator associated with the eigenvalue $2n-1$. We denote by $\overline\omega_j=d^{\mathbb{S}^{2n-1}}\overline\phi_j$ and $\overline{\hat\omega_j}$ its harmonic extension given by \eqref{eq:harmonicextensionomega}. 
By using the fact that $\mathcal{L}_X(f(r)\beta)=f'(r)\beta+f(r)\mathcal{L}^{\mathbb{S}^{2n-1}}_X\beta$, for any differential form $\beta$ in $\mathbb{S}^{2n-1}$ and any vector field $X$ in $\mathbb{R}^{2n}$ tangent to the sphere, we compute 
\begin{eqnarray*} 
\mathcal{L}_\eta\overline{\hat\omega_j}&=&\frac{2n-1}{2n}\Big(\left(-2r\eta(r)\overline\phi_j+(1-r^2)\eta(\overline\phi_j)\right) dr+(1-r^2)\overline\phi_j\mathcal{L}_\eta dr\\&&+\left(\eta(r)+\frac{3r^2}{2n-1}\eta(r)\right)d^{\mathbb{S}^{2n-1}}\overline\phi_j+(r+\frac{1}{2n-1}r^3) d^{\mathbb{S}^{2n-1}}(\eta(\overline\phi_j))\Big)\\
&=&-\frac{(2n-1)i}{2n}\left((1-r^2)\overline\phi_j dr+(r+\frac{1}{2n-1}r^3)d^{\mathbb{S}^{2n-1}}\overline\phi_j)\right)\\
&=&-i\overline{\hat\omega_j},
\end{eqnarray*}
where, in the second equality, we use $\eta(\overline\phi_j)=-i\overline\phi_j$ and $\eta(r)=0$, which can be proven by a straightforward computation.  
Hence, we deduce that 
$\langle\mathcal{L}_\eta\overline{\hat\omega_j}, \overline{\hat\omega_j}\rangle=-i|\hat\omega_j|^2$
which clearly has a negative imaginary part. 
\end{proof}

\section{Computation of the spectrum} 
\label{sec:examples}

In this section we present some explicit computations for the spectrum of the magnetic Steklov operator on differential forms on spheres and balls for certain magnetic potentials.

\subsection{Case of the magnetic Hodge Laplacian on $1$- and $3$-spheres}
In this section, we compute the spectrum of the the magnetic Hodge Laplacian on $\mathbb{S}^1$ and on $\mathbb{S}^3$ associated to the magnetic potential given by a Killing vector field. We first begin with the circle $\mathbb{S}^1$.   

\begin{theorem} Let $M=\mathbb{S}^1$ be the unit circle equipped with the standard metric and consider the magnetic potential  $t\eta$, for $t>0$, where $\eta=d\theta$ is the unit parallel $1$-form on $\mathbb{S}^1$. The spectrum of the magnetic Hodge Laplacian $\Delta^{t\eta}$ on $1$-forms is equal to $(k\pm t)^2$, where $k\in \mathbb{Z}$. In particular, the first magnetic Betti number $b_1^{k\eta}(M)=1$. 
\end{theorem}
\begin{proof} It is shown in \cite[Ex. 2.3]{ELMP:16} that the spectrum of the magnetic Laplacian operator restricted to smooth complex functions is given by $(k\pm t)^2$ for $k\in \mathbb{Z}$. The eigenfunctions are of the form $e^{\pm ik\theta}$. Since the magnetic Hodge Laplacian commutes with the Hodge star operator \cite[Cor. 3.2]{EGHP:23}, the spectrum of this operator restricted to $1$-forms reduces to the one of the magnetic Laplacian on functions. Therefore, we deduce that the spectrum is $(k\pm t)^2$ for $k>0$. However, the volume form $d\theta$ of $\mathbb{S}^1$ is an eigenform of the magnetic Hodge Laplacian since by the magnetic Bochner formula \cite[Prop.3.6] {EGHP:23} we have 
$$\Delta^{t\eta}d\theta=\Delta^{\mathbb{S}^1}(d\theta)-2it\mathcal{L}_{d\theta}d\theta+t^2d\theta=t^2d\theta.$$
In the last equality, we used the fact that $d\theta$ is a parallel $1$-form on $\mathbb{S}^1$, thus, $\Delta^{\mathbb{S}^1}(d\theta)=0$ and $\mathcal{L}_{d\theta}d\theta=0$. 

Finally, when $t=k$, the lowest eigenvalue is zero of multiplicity one which means that the first magnetic Betti number $b_1^{k\eta}(M):={\rm dim}({\rm Ker}(\Delta^{k\eta}))$ defined in \cite{EGHP:23} is equal to $1$. This finishes the proof of the theorem.
\end{proof}
In the following, we reproduce the example of the round sphere $M=\mathbb{S}^3$ equipped with the standard metric $g$ of curvature equal to $1$ where the magnetic potential is given by the unit Killing vector field that defines the Hopf fibration. We refer to \cite{EGHP:23} for more details.  The three vectors 
\begin{eqnarray}\label{eq:etaform}
Y_2 &=& -y_1 \partial_{x_1} + x_1 \partial_{y_1} - y_2 \partial_{x_2} + x_2 \partial_{y_2}, \\
Y_3 &=& -y_2 \partial_{x_1} - x_2 \partial_{y_1} + y_1 \partial_{x_2} + x_1 \partial_{y_2}, \nonumber\\
Y_4 &=& x_2 \partial_{x_1} - y_2 \partial_{y_1} - x_1 \partial_{x_2} + y_1 \partial_{y_2}\nonumber
\end{eqnarray}
are Killing and form an orthonormal basis of $T_{(z_1,z_2)}\mathbb{S}^3$ at every point $(z_1,z_2) = (x_1 + iy_1 , x_2 + iy_2) \in \mathbb{S}^3\subset\mathbb{C}^2$. Recall that $Y_2$ spans the vertical space of the Hopf fibration. A straightforward computation shows that these vector fields satisfy 
$$[Y_2,Y_3]=-2Y_4, \,\,\,\,[Y_2,Y_4]=2 Y_3,\, \,\,\, [Y_3,Y_4]=-2 Y_2.$$
The Christoffel symbols of the Levi-Civita connection of $g$ are expressed as
\begin{equation*} 
\nabla_{Y_j}^{\mathbb{S}^3} Y_k= \Gamma_{jk} Y_l
\end{equation*}
with $\{j,k,l\} = \{2,3,4\}$ for $k \neq j$, $\Gamma_{jj} = 0$ and $\Gamma_{23}=-\Gamma_{24} =-1$, $\Gamma_{32}=\Gamma_{43} = - \Gamma_{34} = -\Gamma_{42} =1$. Hence, we get that 
\begin{equation} \label{eq:covYiYj}
d^{{\mathbb{S}^3}} Y_2=2 Y_3\wedge Y_4,\, d^{{\mathbb{S}^3}} Y_3=-2 Y_2\wedge Y_4,\, d^{{\mathbb{S}^3}} Y_4=2 Y_2\wedge Y_3,\,\, \delta^{{\mathbb{S}^3}} Y_j=0,
\end{equation}
for $j=1,2,3$. 

In what follows, we consider the magnetic vector field $tY_2$, for $t>0$. We have shown in \cite{EGHP:23} that the spectrum of the corresponding magnetic Laplacian $\Delta^{tY_2}$ on complex  smooth functions is given by
\begin{equation} \label{eq:specmagS3}
k(k+2)+2(2p-k)t+t^2, \quad k\in\NN\cup\{0\},\,\, p\in\{0,\ldots, k\}.
\end{equation}
with multiplicities $(k+1)^2$ and the corresponding eigenfunctions are $f_{p,k}:=u^pv^{k-p}$, where $u=az_1+bz_2$ and $v=b\overline{z}_1-a\overline{z}_2$ for $(a,b)\in \mathbb{C}^2\setminus\{(0,0)\}$. The functions $f_{p,k}$ are the restrictions of homogeneous harmonic polynomials on $\mathbb{C}^2$ of degree $k$ to the unit sphere  $\mathbb{S}^3$. A straightforward computation (see, e.g., \cite[p. 30]{Hi:74} or \cite[Lemma III.7.1]{Pe:93}) yields
\begin{equation}\label{eq:Y2phi} 
Y_2(f_{p,k})=i(2p-k)f_{p,k}. 
\end{equation}
Also, one can easily check from Equations \eqref{eq:covYiYj} that $d^{2Y_2}(Y_3-iY_4)=0$ and $\delta^{2Y_2}(Y_3-iY_4)=0$, which means that $Y_3-iY_4$ is a non-vanishing element in ${\rm Ker}(\Delta^{2Y_2})$. Hence the first magnetic Betti number $b_1^{2Y_2}(M)$ is at least one. 

In what follows, we will prove that $b_1^{tY_2}(M)$ is equal to $1$, for $t\in \mathbb{N}, t\geq 2$, by computing the full spectrum of the magnetic Hodge Laplacian. Notice that this operator commutes with the Hodge star operator \cite[Cor. 3.2]{EGHP:23},  thus, it is sufficient to compute the spectrum on $1$-forms. To simplify the notation, we denote the vector field $Y_2$ by $\eta$.  

\begin{theorem} \label{thm:spherehodge} Let $(\mathbb{S}^3,g)$ be the round sphere equipped with the standard metric and consider the magnetic potential  $t\eta$, for $t>0$, where $\eta$ is the Killing vector field that defines the Hopf fibration. The spectrum of the magnetic Hodge Laplacian $\Delta^{t\eta}$ on $1$-forms consists of eigenvalues on exact and co-exact forms. On the family of exact $1$-forms, the eigenvalues are given by 
\begin{equation*} 
k(k+2)+2(2p-k)t+t^2, \quad k\geq 1,\,\, p\in\{0,\ldots, k\}.
\end{equation*}
with multiplicities $k(k+2)$. 
On the family of co-exact $1$-forms, the eigenvalues are given by 
$$(k+1)^2+2t(2p-k\pm 1)+t^2,$$
for $k>0,\, p\in\{0,\ldots, k\}$ with multiplicities $k(k+2)$. 
In particular, for $t=k+1$, the lowest eigenvalue is equal to $0$. Therefore, the first Betti numbers $b_1^{t\eta}(M)=1$, for $t\in \mathbb{N}, t\geq 2$.
\end{theorem}

In Figure \ref{fig:magnlapS3eigval}, we give the graph of the eigenvalues of the magnetic Hodge Laplacian for $1$-forms on $\mathbb{S}^3$ as well as the graph of the first eigenvalue $\lambda_{1,1}^{t\eta}(\mathbb{S}^3)$ (as functions in the parameter $t$). We see that the first eigenvalue $\lambda_{1,1}^{t\eta}(\mathbb{S}^3) $ is less than $\lambda_{1,1}(\mathbb{S}^3)=3$, for all $t>0$. 

\begin{figure}[ht]
\centering

\begin{minipage}{.49\textwidth}
  \includegraphics[width= \textwidth]{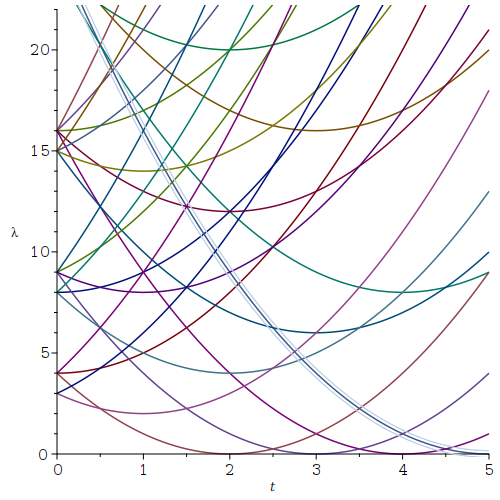}
\end{minipage}
\begin{minipage}{0.49\textwidth}
  \includegraphics[width= \textwidth]
  {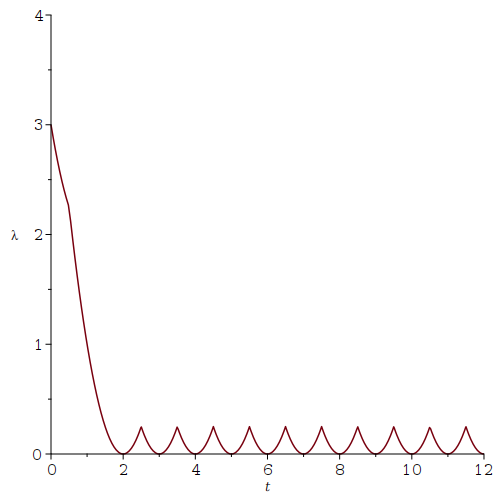}
\end{minipage}  
  \caption{Eigenvalues of $1$-forms of the magnetic Hodge Laplacian $\Delta^{t \eta}$ on $\mathbb{S}^3$ as functions over $t \in [0,5]$ (left) and first eigenvalue $\lambda^{t\eta}_{1,1}(\mathbb{S}^3)$ as a function over $t \in [0,12]$ (right).} \label{fig:magnlapS3eigval}
\end{figure}
\FloatBarrier
\begin{proof} Since $\Delta^{t\eta}$ commutes with the exterior differential $d^{{\mathbb{S}^3}}$ and with $\delta^{{\mathbb{S}^3}}$ (see \cite[Prop. 3.6]{EGHP:23}), it then commutes with $\Delta^{\mathbb{S}^3}$. Thus, we can consider eigenforms for $\Delta^\eta$ that are exact or co-exact which are eigenforms of the Hodge Laplacian $\Delta^{\mathbb{S}^3}$. The spectrum of $\Delta^{t\eta}$ restricted to the set of exact forms is given by \eqref{eq:specmagS3} with corresponding eigenforms $d^{{\mathbb{S}^3}}f_{p,k}$ with $k\geq 1$, since for $k=0$, we have $p=0$ and hence the function $f_{p,k}$ is constant. 

Now, to compute the spectrum of $\Delta^{t\eta}$ on co-exact $1$-forms, it is sufficient to  compute the spectrum on exact (or closed) $2$-forms. Since the magnetic Laplacian $\Delta^{t\eta}$ commutes with the Hodge Laplacian $\Delta^{\mathbb{S}^3}$, the exterior differential $d^{\mathbb{S}^3}$ and the Lie derivative $\mathcal{L}_\eta$ \cite[Prop. 3.6]{EGHP:23}, we can restrict the computation of the spectrum to a closed eigenform of the Hodge Laplacian which is also an eigenform of the Lie derivative. For this, we recall that the eigenvalues of the Hodge Laplacian on closed $2$-forms are equal to $(k+1)^2, k>0$ (see \cite[Prop. 2.1]{Pa:79} and \cite{GM:75}), with multiplicities $2k(k+2)$ and the eigenforms are of the form $\iota^*d\omega$, where $\omega=\sum_{j=1}^4 \omega_j dx^j$, $\omega_j$ are homogeneous harmonic polynomials of degree $k$ in $\mathbb{R}^4$, and $d$ is the exterior differential in $\mathbb{R}^4$. Here $\iota: \mathbb{S}^3\to \mathbb{R}^4$ is the inclusion map. Hence, by setting $dz^j=dx^j+idy^j$ for all $j=1,\ldots, 4$, the forms $\iota^*d(f_{p,k}dz^j)$ and $\iota^*d(f_{p,k}d\overline{z}^j)$ are eigenforms of the Hodge Laplacian and, therefore, we compute 
\begin{eqnarray}\label{eq:liederivativesphere}
\mathcal{L}_{\eta} (\iota^*d(f_{p,k}dz^j))&=&\iota^*d (\mathcal{L}_{\eta}(f_{p,k}dz^j))\nonumber\\
&=&\iota^*d \left(\eta(f_{p,k}) dz^j+f_{p,k}d(\eta(z^j))\right)\nonumber\\
&\stackrel{\eqref{eq:Y2phi}}{=}&\iota^*d(i(2p-k)f_{p,k} dz^j+if_{p,k}dz^j)\nonumber\\
&=&i(2p-k+1)\iota^*d(f_{p,k}dz^j).
\end{eqnarray}
In the third equality, we use the fact that $\eta(z^j)=iz^j$ which can be proven straightforwardly using the definition of $\eta$ in \eqref{eq:etaform}. In the same way, we get $\mathcal{L}_{\eta} (\iota^*d(f_{p,k}d\overline{z}^j))=i(2p-k-1)\iota^*d(f_{p,k}d\overline{z}^j)$.  Hence using the formula $\Delta^{t\eta}=\Delta^{\mathbb{S}^3}-2it\mathcal{L}_\eta+t^2|\eta|^2$ valid on $p$-forms \cite[Prop. 3.6]{EGHP:23}, we get that 
\begin{equation*}
 \Delta^{t\eta}(\iota^*d(f_{p,k}dz^j))=((k+1)^2+2t(2p-k+1)+t^2)\iota^*d(f_{p,k}dz^j).
 \end{equation*}
In the same way, we have 
\begin{equation*}
 \Delta^{t\eta}(\iota^*d(f_{p,k}d\overline{z}^j))=((k+1)^2+2t(2p-k-1)+t^2)\iota^*d(f_{p,k}d\overline{z}^j).
 \end{equation*}
Hence the spectrum of $\Delta^{t\eta}$ on closed $2$-forms on $\mathbb{S}^3$ is given by the families  
$$(k+1)^2+2t(2p-k\pm 1)+t^2,\,\,\, $$
for $k>0,\, p\in\{0,\ldots, k\}$, of multiplicities $k(k+2)$. The second part of the theorem is a direct consequence of the computation of the spectrum by taking $t=k+1$ and $p=0$. 
\end{proof}
 
\subsection{Computation of the spectrum of the magnetic Steklov operator on the Euclidean unit ball in $\mathbb{R}^2$}
In this section, we compute the spectrum of the magnetic Steklov operator with a certain magnetic potential given by a Killing vector field on the Euclidean unit ball in $\mathbb{R}^2$. 

\begin{theorem} Let $({\mathbb B}^2,g)$ be the Euclidean unit ball equipped with the metric $g=dr^2\oplus r^2 d\theta^2$ and let $\eta:=-y\partial_{x}+x \partial_{y}$ be the  Killing vector field in $\mathbb{R}^2$. The spectrum of the magnetic Steklov operator $T^{[1],t\eta}$ on differential $1$-forms on $\mathbb{S}^1$ associated with the magnetic potential $t\eta$ ($t>0$) is given by  
$$t\frac{{\rm cosh}(t/2)}{{\rm sinh}(t/2)}, \quad \frac{(-t)^{k+1}}{k!\left(e^{-t}-\sum_{j=0}^k\frac{(-t)^{j}}{j!}\right)},\,\,\, \frac{t^{k+1}}{k!\left(e^{t}-\sum_{j=0}^k\frac{t^{j}}{j!}\right)} \, \text{ for } k \geq 1.$$
In particular, the lowest eigenvalue $\frac{t^2}{e^t-1-t}$ is less than $\sigma_{1,1}(\mathbb{B}^2)=2$.
\end{theorem}

The eigenvalues of the magnetic Steklov operator $T^{[1],t\eta}$ 
acting on $1$-forms on $\mathbb{S}^1$ are illustrated in Figure \ref{fig:SteklovS11forms}.

\begin{figure}[ht]
\centering
  \includegraphics[width= 0.6\textwidth]{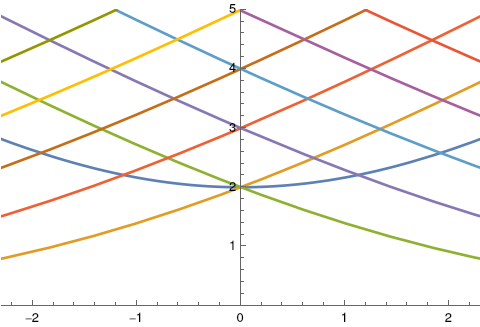}
  \caption{Eigenvalues of the magnetic Steklov operator $T^{[1],t\eta}$ on $1$-forms in $\mathbb{S}^1$.} \label{fig:SteklovS11forms}
\end{figure}
\FloatBarrier

We follow the same strategy of proof that was used in \cite{RS:14} for the non-magnetic case, and we must deal with the additional terms that arise in the magnetic setting.

\begin{proof} It is not difficult to check that the form $e^{ik\theta}d\theta$ is an eigenform of the Laplacian associated with the eigenvalue $k^2$, for $k\geq 0$.
We aim to find an $\eta$-harmonic extension of the differential form $\omega=ie^{ik\theta}d\theta$ on $\mathbb{S}^1$ of the form $\hat\omega=iQ e^{ik\theta}d\theta+P e^{ik\theta}dr$, where $P$ and $Q$ are smooth functions in $r$ to be determined. First we compute 
\begin{eqnarray*}
\mathcal{L}_\eta\hat\omega&=&\eta\lrcorner d(iQ e^{ik\theta}d\theta+P e^{ik\theta}dr)+d(\eta\lrcorner (iQ e^{ik\theta}d\theta+Pe^{ik\theta}dr)) \\
&=&\eta\lrcorner ((iQ'-ikP)e^{ik\theta} dr\wedge d\theta)+id(Qe^{ik\theta})\\
&=&-i(Q'-kP)e^{ik\theta} dr+iQ'e^{ik\theta}dr-kQe^{ik\theta}d\theta\\
&=&ik\hat\omega.
\end{eqnarray*}
Now, we use \cite[Prop. 3.6]{EGHP:23} and \cite[Lem. 6]{RS:14} to compute the magnetic Laplacian for any $t>0$, 
\begin{eqnarray*}
\Delta^{t\eta}\hat\omega&=&\Delta^{\mathbb{R}^2}\hat \omega-2it\mathcal{L}_\eta\hat\omega+t^2r^2\hat\omega\\
&=&\frac{Q}{r^2}\Delta^{\mathbb{S}^1}(ie^{ik\theta}d\theta)-(Q''-\frac{1}{r}Q')ie^{ik\theta} d\theta-\frac{2P}{r}d^{\mathbb{S}^1}(e^{ik\theta})\\&&+\left(\frac{P}{r^2}\Delta^{\mathbb{S}^1}(e^{ik\theta})-(P'+\frac{1}{r}P)'e^{ik\theta}-\frac{2Q}{r^3}\delta^{\mathbb{S}^1}(ie^{ik\theta}d\theta)\right)dr\\
&&+(2kt+t^2r^2)(iQ e^{ik\theta}d\theta+P e^{ik\theta}dr)\\
&=&\left(\frac{k^2Q}{r^2}-(Q''-\frac{1}{r}Q')-\frac{2kP}{r}+(2kt+t^2 r^2)Q\right)ie^{ik\theta}d\theta\\
&&+\left(\frac{k^2 P}{r^2}-(P'+\frac{1}{r}P)'-\frac{2kQ}{r^3}+(2kt+t^2 r^2)P\right)e^{ik\theta}dr.
\end{eqnarray*}
Hence $\hat\omega$ is $\eta$-harmonic if and only if $P$ and $Q$ solve the differential system 
\begin{equation}\label{eq:systemdiff}
\left\{
\begin{matrix}
 \frac{k^2Q}{r^2}-(Q''-\frac{1}{r}Q')-\frac{2kP}{r}+(2kt+t^2 r^2)Q=0,\\\\
\frac{k^2P}{r^2}-(P'+\frac{1}{r}P)'-\frac{2kQ}{r^3}+(2kt+t^2 r^2)P=0,
\end{matrix}\right.
\end{equation}
with the initial conditions $P(0)=Q(0)=0$ and $P'(0)=Q'(0)=0$. 

The condition $\nu\lrcorner \hat\omega=0$ on $\mathbb{S}^1$ is equivalent to $P(1)=0$.  Bear in mind that, for $t=0$, the solution of the system \eqref{eq:systemdiff} reduces to the one without magnetic potential and is given, for $k\geq 1$, by $P(r)=r^{k-1}(1-r^2)$ and $Q(r)=r^k(1+r^2)$. 

In order to completely solve the system, we first consider the case when $k=0$. In this case, it is not difficult to check that the only solution satisfying $P(1)=0$ and $P'(0)=0$ is $P=0$ and $Q=\frac{{\rm sinh}(\frac{tr^2}{2})}{{\rm sinh}(\frac{t}{2})}$. Here, we use the normalising factor $Q(1)=1$.

For $k\geq 1$, we consider the following change of variables $P:=r^{k-1} \hat P$ and $Q:=r^{k}\hat Q$, where $\hat P$ and $\hat Q$ are two smooth functions in $r$ such that $\hat P(0)=1$ and $\hat Q(0)=1$ to be determined. By a straightforward computation and after replacing $P$ and $Q$, the above system becomes 
\begin{equation}\label{eq:systemdiffmodi}
\left\{
\begin{matrix}
 (2k+2ktr^2+t^2r^4)\hat Q+r(1-2k)\hat Q'-r^2\hat Q''-2k\hat P=0\\\\
 (2k+2ktr^2+t^2r^4)\hat P+r(1-2k)\hat P'-r^2\hat P''-2k\hat Q=0
\end{matrix}\right.
\end{equation}
By setting $\hat Z:=\hat P+\hat Q$ and $\hat W:=\hat P-\hat Q$, the system  \eqref{eq:systemdiffmodi}  becomes 
\begin{equation}\label{eq:systemdiffmodi2}
\left\{
\begin{matrix}
(2ktr^2+t^2r^4)\hat Z+r(1-2k)\hat Z'-r^2\hat Z''=0\\\\
 (4k+2ktr^2+t^2r^4)\hat W+r(1-2k)\hat W'-r^2\hat W''=0,
\end{matrix}\right.
\end{equation}
with the initial condition $\hat Z(0)=2, \hat Z'(0)=0, \hat W(0)=0, \hat W'(0)=0$. Now, one can check that $\hat Z=2 e^{tr^2/2}$ and that 
$$\hat W=c\frac{(-1)^k k!e^{-tr^2/2}}{r^{2k}}\left(e^{tr^2}\sum_{j=0}^k\frac{(-tr^2)^j}{j!}-1\right )$$
is a solution of \eqref{eq:systemdiffmodi2}, where $c$ is a constant. Hence, for $k\geq 1$, we deduce the following expressions for $\hat P$ and $\hat Q$,  
$$\hat P=e^{tr^2/2}+c\frac{(-1)^k k!e^{-tr^2/2}}{2r^{2k}}\left(e^{tr^2}\sum_{j=0}^k\frac{(-tr^2)^j}{j!}-1\right ),$$
and 
$$\hat Q=e^{tr^2/2}-c\frac{(-1)^k k!e^{-tr^2/2}}{2r^{2k}}\left(e^{tr^2}\sum_{j=0}^k\frac{(-tr^2)^j}{j!}-1\right ).$$
The condition $P(1)=0$ is equivalent to  $\hat P(1)=0$ which is equivalent to saying 
\begin{equation}\label{eq:constantc}
-2=c(-1)^k k!\left(\sum_{j=0}^k\frac{(-t)^j}{j!}-e^{-t}\right).
\end{equation}
Now, to check that the right-hand side of \eqref{eq:constantc} does not vanish,  we just use the Taylor-Lagrange  formula for $e^{-t}$ which says that, there exists some real number $a_{t}\in [0,t]$ such that 
\begin{equation}\label{eq:taylorexp}
e^{-t}=\sum_{j=0}^k\frac{(-t)^j}{j!}+\frac{(-1)^{k+1}e^{-a_{t}}}{(k+1)!}t^{k+1}.
\end{equation}
Hence, we deduce that 
$c=\frac{-2}{(-1)^k k!\left(\sum_{j=0}^k\frac{(-t)^j}{j!}-e^{-t}\right)}$.
Thus, for $k\geq 1$, after dividing by $Q(1)$, $Q$ is equal to  
$$Q(r)=\frac{r^ke^{t(r^2-1)/2}}{2}+\frac{e^{t(r^2-1)/2}\left(\sum_{j=0}^k\frac{(-tr^2)^{j}}{j!}-e^{-tr^2}\right)}{2r^{k}\left(\sum_{j=0}^k\frac{(-t)^{j}}{j!}-e^{-t}\right)}.$$

Next, we want to find the $\eta$-harmonic extension of $ie^{-ik\theta}$, which takes the form $Q ie^{-ik\theta}d\theta+P e^{-ik\theta}dr$. A similar computation as before shows that the corresponding differential system is 
\begin{equation*}
\left\{
\begin{matrix}
 \frac{k^2Q}{r^2}-(Q''-\frac{1}{r}Q')+\frac{2kP}{r}+(-2kt+t^2 r^2)Q=0\\\\
\frac{k^2P}{r^2}-(P'+\frac{1}{r}P)'+\frac{2kQ}{r^3}+(-2kt+t^2 r^2)P=0,
\end{matrix}\right.
\end{equation*}
with the initial conditions $P(0)=Q(0)=0$ and $P'(0)=Q'(0)=0$. As we did before, we set $P:=r^{k-1} \hat P$ and $Q:=r^{k}\hat Q$, where $\hat P$ and $\hat Q$ are two smooth functions in $r$ such that $\hat P(0)=1$ and $\hat Q(0)=-1$. Hence, the above system becomes 
\begin{equation*}
\left\{
\begin{matrix}
 (2k-2ktr^2+t^2r^4)\hat Q+r(1-2k)\hat Q'-r^2\hat Q''+2k\hat P=0\\\\
 (2k-2ktr^2+t^2r^4)\hat P+r(1-2k)\hat P'-r^2\hat P''+2k\hat Q=0.
\end{matrix}\right.
\end{equation*}
Therefore, by adding and subtracting these two equations, after setting $\hat Z=\hat P+\hat Q$ and $\hat W=\hat P-\hat Q$, the above system becomes equivalent to 
\begin{equation*}
\left\{
\begin{matrix}
 (4k-2ktr^2+t^2r^4)\hat Z+r(1-2k)\hat Z'-r^2\hat Z''=0\\\\
(-2ktr^2+t^2r^4)\hat W+r(1-2k)\hat W'-r^2\hat W''=0
\end{matrix}\right.
\end{equation*}
with the initial condition $\hat Z(0)=0, \hat Z'(0)=0, \hat W(0)=2, \hat W'(0)=0$. This is exactly the same system as \eqref{eq:systemdiffmodi2} by switching the parameter $t$ to $-t$. Thus the rest of the proof follows the same steps as in the previous case. Hence we get 
$$Q(r)=\frac{r^ke^{-t(r^2-1)/2}}{2}+\frac{e^{-t(r^2-1)/2}\left(\sum_{j=0}^k\frac{(tr^2)^{j}}{j!}-e^{tr^2}\right)}{2r^{k}\left(\sum_{j=0}^k\frac{t^{j}}{j!}-e^{t}\right)}.$$

Finally, to compute the eigenvalues of the magnetic Steklov operator in both cases, we apply it to $\omega=ie^{ik\theta} d\theta$ (resp. $ie^{-ik\theta} d\theta$) and use its $\eta$-harmonic extension to get  
\begin{eqnarray*}
T^{[1],t\eta}\omega&=&\partial r\lrcorner d^{t\eta}\hat\omega\\
&=&\partial r\lrcorner d(iQ e^{ik\theta}d\theta+P e^{ik\theta}dr)+i\partial r\lrcorner\left(t\eta\wedge (iQ e^{ik\theta}d\theta+P e^{ik\theta}dr)\right)\\
&=&Q'(1)\omega.
\end{eqnarray*}
In the last equality, we use the fact that $\eta=d\theta$ on $\mathbb{S}^1$. Thus, the eigenvalues are equal to 
$$\frac{(-t)^{k+1}}{k!\left(e^{-t}-\sum_{j=0}^k\frac{(-t)^{j}}{j!}\right)}, \,\, \left(\textrm{resp.}\,\,\, \frac{t^{k+1}}{k!\left(e^{t}-\sum_{j=0}^k\frac{t^{j}}{j!}\right)}\right)$$
for $k\geq 1$, and for $k=0$, to $t\frac{{\rm cosh}(t/2)}{{\rm sinh}(t/2)}$. 
\end{proof}

\subsection{Computation of the spectrum of the magnetic Steklov operator on the Euclidean unit ball in $\mathbb{R}^4$}

In this section, we compute the spectrum of the magnetic Steklov operator with a certain magnetic potential given by a Killing vector field on $\mathbb{B}^4$. We treat the case of $1$-forms, since the other degrees can be done in the same way. 

\begin{theorem} Let $(\mathbb{B}^4,g)$ be the Euclidean unit ball equipped with the metric $g=dr^2\oplus r^2 d\theta^2$ and let $\eta:=-y_1 \partial_{x_1} + x_1 \partial_{y_1} - y_2 \partial_{x_2} + x_2 \partial_{y_2}$ be the Killing vector field in $\mathbb{R}^4$. The spectrum of the magnetic Steklov operator $T^{[1],t\eta}$ restricted to exact $1$-eigenforms of the Hodge Laplacian on $\mathbb{S}^3$ is given by 
$$ \left(\sigma_{(k,p),1}^{t \eta}\right)'=\frac{1}{k+1}\left[ k\left(p+\frac{3}{2}\right)\frac{L_{k-\frac{1}{2}-p}^{(-(k+2))}(t)}{L_{k+\frac{1}{2}-p}^{(-(k+2))}(t)} + (k+2)\left(p+\frac{1}{2}\right)\frac{L_{k-\frac{3}{2}-p}^{(-k)}(t)}{L_{k-\frac{1}{2}-p}^{(-k)}(t)} + k^2-(2p+t)(k+1)-1 \right], $$
for $k\geq 1$ and $p\in\{0,\ldots,k\}$. The spectrum of the magnetic Steklov operator restricted to co-exact $1$-eigenforms of the Hodge Laplacian on $\mathbb{S}^3$ is given by
$$ \left( \sigma_{(k,p),1,\pm}^{t \eta}\right)''=\frac{-2tL^{(-k)}_{k-1\pm \frac{1}{2}-p}(t)}{L^{(-(k+1))}_{k\pm\frac{1}{2}-p}(t)}-(k+t+1)
$$
The lowest eigenvalue is equal to 
$$-\frac{3t L^{1-k}_{-1/2}(t)}{2L^{-1}_{-1/2}(t)}-\frac{t L^{-2}_{1/2}(t)}{2L^{-3}_{3/2}(t)}-\frac{2t+3}{2},$$
which is less than $\sigma_{1,1}(\mathbb{B}^4)=\frac{3}{2}$, for $t$ small $(0<t<2.99)$. 
\end{theorem} 
\begin{proof}
We consider $\omega$ an exact $1$-form on $\mathbb{S}^3$, which is an eigenform of the Laplacian $\Delta^{\mathbb{S}^3}$. For $k\geq 1$ and $p\in\{0,\ldots,k\}$, we write $\omega=d^{\mathbb{S}^3}f_{p,k}$ where $f_{p,k}$ are as defined previously. 
As we did before, we need to find an $\eta$-harmonic extension of $\omega$ which has the form $\hat\omega=Q d^{\mathbb{S}^3}f_{p,k}+Pf_{p,k}dr$, where $P$ and $Q$ are two smooth functions in $r$ to be determined. With the help of \eqref{eq:Y2phi}, we have 
\begin{eqnarray*}
\mathcal{L}_\eta\hat\omega&=&Q d^{\mathbb{S}^3}(\eta(f_{p,k}))+P\eta(f_{p,k})dr\\
&=&i(2p-k)Q d^{\mathbb{S}^3}f_{p,k}+i(2p-k)Pf_{p,k}dr\\
&=&i(2p-k)\hat\omega.
\end{eqnarray*}
Using \cite[Prop. 3.6]{EGHP:23} and \cite[Lem. 6]{RS:14}, for any $t>0$, we compute  
\begin{eqnarray*}
\Delta^{t\eta}\hat\omega&=&\Delta^{\mathbb{R}^4}\hat \omega-2it\mathcal{L}_\eta\hat\omega+t^2r^2\hat\omega\\
&=&\frac{Q}{r^2}\Delta^{\mathbb{S}^3}(d^{\mathbb{S}^3} f_{p,k})-(Q''+\frac{1}{r}Q')d^{\mathbb{S}^3} f_{p,k}-\frac{2P}{r}d^{\mathbb{S}^3}f_{p,k}\\&&+\left(\frac{P}{r^2}\Delta^{\mathbb{S}^3}(f_{p,k})-(P'+\frac{3}{r}P)'f_{p,k}-\frac{2Q}{r^3}\delta^{\mathbb{S}^3}(d^{\mathbb{S}^3}f_{p,k})\right)dr\\
&&+(2(2p-k)t+t^2r^2)(Q d^{\mathbb{S}^3}f_{p,k}+Pf_{p,k}dr)\\
&=&\left(\frac{k(k+2)Q}{r^2}-(Q''+\frac{1}{r}Q')-\frac{2P}{r}+(2(2p-k)t+t^2 r^2)Q\right)d^{\mathbb{S}^3}f_{p,k}\\
&&+\left(\frac{k(k+2) P}{r^2}-(P'+\frac{3}{r}P)'-\frac{2k(k+2)Q}{r^3}+(2(2p-k)t+t^2 r^2)P\right)f_{p,k}dr.
\end{eqnarray*}
Hence $\hat\omega$ is $\eta$-harmonic if and only if $P$ and $Q$ solve the differential system 
\begin{equation*}
\left\{
\begin{matrix}
\frac{k(k+2)Q}{r^2}-(Q''+\frac{1}{r}Q')-\frac{2P}{r}+(2(2p-k)t+t^2 r^2)Q=0,\\\\
\frac{k(k+2)P}{r^2}-(P'+\frac{3}{r}P)'-\frac{2k(k+2)Q}{r^3}+(2(2p-k)t+t^2 r^2)P=0,
\end{matrix}\right.
\end{equation*}
The condition $\nu\lrcorner \hat\omega=0$ on $\mathbb{S}^1$ is equivalent to $P(1)=0$. By setting $P=r^{k-1}\hat P$ and $Q=r^k\hat Q$, with $\hat P(0)=k$ and $\hat Q(0)=1$, the system becomes
\begin{equation*}
\left\{
\begin{matrix}
(2k+2(2p-k)tr^2+t^2r^4)\hat Q-(2k+1)r\hat Q'-r^2\hat Q''-2\hat P=0\\\\
(2k+4+2(2p-k)tr^2+t^2r^4)\hat P-(2k+1)r\hat P'-r^2\hat P''-2k(k+2)\hat Q=0
\end{matrix}\right.
\end{equation*}
Setting $\hat Z:=\frac{\hat P}{k+2}+\hat Q$ and $\hat W:=\frac{\hat P}{k+2}-\frac{k}{k+2}\hat Q$, we see that $\hat Z$ and $\hat W$ satisfy the following differential equations: 
\begin{equation*}
\left\{
\begin{matrix}
\hat Z''+\frac{(2k+1)}{r}\hat Z'-(2(2p-k)t+t^2r^2)\hat Z=0\\\\
\hat W''+\frac{(2k+1)}{r}\hat W'-(\frac{4k+4}{r^2}+2(2p-k)t+t^2r^2)\hat W=0,
\end{matrix}\right.
\end{equation*}
with the initial conditions $\hat Z(0)=\frac{2k+2}{k+2}, \hat Z'(0)=0$ and $\hat W(0)=0, \hat W'(0)=0$. The solution is given by 
$$\hat Z=\frac{2(-1)^k(k+1)!e^{-tr^2/2}L^{-k}_{k-1/2-p}(tr^2)}{(k+2)t^kr^{2k}},$$ 
and 
$$\hat W=c\frac{e^{-tr^2/2}L^{-(k+2)}_{k+1/2-p}(tr^2)}{r^{2k+2}},$$ 
where $c$ is a constant. Hence, we get that 
$$\hat P=\frac{(-1)^kk k!e^{-tr^2/2}L^{-k}_{k-1/2-p}(tr^2)}{t^kr^{2k}}+\frac{c(k+2)^2e^{-tr^2/2}L^{-(k+2)}_{k+1/2-p}(tr^2)}{2(k+1)r^{2k+2}}$$
and 
$$\hat Q=\frac{(-1)^kk!e^{-tr^2/2}L^{-k}_{k-1/2-p}(tr^2)}{t^k r^{2k}}-c\frac{(k+2)e^{-tr^2/2}L^{-(k+2)}_{k+1/2-p}(tr^2)}{2(k+1)r^{2k+2}}.$$
Now, $\hat P(1)=0$ gives that $c=\frac{2(-1)^{k+1}k(k+1)!L^{-k}_{k-1/2-p}(t)}{t^k(k+2)^2L^{-(k+2)}_{k+1/2-p}(t)}$. Hence, we deduce the following expression for $Q$ as: 
$$Q=\frac{(k+2)e^{t(1-r^2)/2}L^{-k}_{k-1/2-p}(tr^2)}{2(k+1)r^{k}L^{-k}_{k-1/2-p}(t)}+\frac{k e^{t(1-r^2)/2}L^{-(k+2)}_{k+1/2-p}(tr^2)}{2(k+1)L^{-(k+2)}_{k+1/2-p}(t)r^{k+2}}.$$
To obtain the eigenvalues on exact $1$-forms of the form $\omega=d^{\mathbb{S}^3}f_{p,k}$, we compute 
$$T^{[1],t\eta}\omega=\partial r\lrcorner d(Q d^{\mathbb{S}^3}f_{p,k}+Pf_{p,k}dr)+i\partial r\lrcorner(\eta\wedge (Q d^{\mathbb{S}^3}f_{p,k}+Pf_{p,k}dr))=Q'(1)\omega$$
For $k\geq 1$, the formula for $Q'(1)$ is
\begin{eqnarray*}
Q'(1)&=&-\frac{k+2}{2(k+1)L^{-k}_{k-1/2-p}(t)}(2t L^{1-k}_{k-3/2-p}(t)+(k+t)L^{-k}_{k-1/2-p}(t))\\&&-\frac{k}{2(k+1)L^{-(k+2)}_{k+1/2-p}(t)}(2tL^{-(k+1)}_{k-1/2-p}(t)+(k+t+2)L^{-(k+2)}_{k+1/2-p}(t))\\
&=&-\frac{(k+2)t L^{1-k}_{k-3/2-p}(t)}{(k+1)L^{-k}_{k-1/2-p}(t)}-\frac{kt L^{-(k+1)}_{k-1/2-p}(t)}{(k+1)L^{-(k+2)}_{k+1/2-p}(t)}-\frac{k^2+kt+2k+t}{k+1}.
\end{eqnarray*}

Next, we consider co-exact eigenforms of the Hodge Laplacian and find their $\eta$-harmonic extensions. Let $\omega$ be a co-exact eigenform associated with the eigenvalue $(k+1)^2$ ($k\geq 1$) and let $\hat\omega=Q(r)\omega$. We have seen in the proof of Theorem \ref{thm:spherehodge} that $\omega$ can be taken to be of the form $*_{\mathbb{S}^3}(\iota^*d(f_{p,k}dz^j))$ or $*_{\mathbb{S}^3}(\iota^*d(f_{p,k}d\overline{z}^j))$, for $j=1,\ldots,4$. Since $\eta$ is Killing, $\eta(r)=0$ (that is, tangent to $\mathbb{S}^3$)  and the Hodge star operator on $\mathbb{S}^3$ commutes with the Lie derivative $\mathcal{L}_\eta$. So, with the help of \eqref{eq:liederivativesphere}, we deduce that $\mathcal{L}_\eta\hat\omega=i(2p-k+1)\hat\omega$ (or $\mathcal{L}_\eta\hat\omega=i(2p-k-1)\hat\omega$, if $\omega$ takes the form $*_{\mathbb{S}^3}(\iota^*d(f_{p,k}d\overline{z}^j))$). Hence, using \cite[Prop. 3.6]{EGHP:23} and \cite[Lem. 6]{RS:14} we compute the magnetic Laplacian of $\hat\omega$ (when $\omega=*_{\mathbb{S}^3}(\iota^*d(f_{p,k}dz^j))$) to get 
\begin{eqnarray*}
    \Delta^{t\eta}\hat\omega&=&\Delta^{\mathbb{R}^4}\hat\omega-2it\mathcal{L}_\eta\hat \omega+t^2 r^2\hat\omega\\
    &=&\frac{Q}{r^2}\Delta^{\mathbb{S}^3}\omega-(Q''+\frac{1}{r}Q')\omega+2t(2p-k+1)\hat\omega+t^2r^2\hat \omega\\
     &=&\left(\frac{Q}{r^2}(k+1)^2-(Q''+\frac{1}{r}Q')+(2t(2p-k+1)+t^2r^2)Q\right)\omega.
\end{eqnarray*}
Therefore, $\hat\omega$ is $\eta$-harmonic if and only if it satisfies the differential equation 
$$Q''+\frac{1}{r}Q'-\left(\frac{(k+1)^2}{r^2}+2t(2p-k+1)+t^2r^2\right)Q=0$$
with $Q(0)=0$ and $Q'(0)=0$. After rescaling, the solution of the above differential equation is given by 
$$Q(r)=e^{t(1-r^2)/2}\frac{L^{-(k+1)}_{k-p-\frac{1}{2}}(tr^2)}{r^{k+1}L^{-(k+1)}_{k-p-\frac{1}{2}}(t)}.$$
As before, one can show that the eigenvalues are given by 
$$Q'(1)=\frac{-2tL^{-k}_{k-3/2-p}(t)}{L^{-(k+1)}_{k-\frac{1}{2}-p}(t)}-(k+t+1).$$
One can also perform a similar computation when $\omega=*_{\mathbb{S}^3}(\iota^*d(f_{p,k}d\overline{z}^j))$ and show that the eigenvalues are equal to 
$\frac{-2tL^{-k}_{k-1/2-p}(t)}{L^{-(k+1)}_{k+\frac{1}{2}-p}(t)}-(k+t+1).$
\end{proof}

{\bf Acknowledgments:} We would like to thank Asma Hassannezhad for many stimulating discussions. We are also grateful to Colette Anné and Luigi Provenzano for helpful comments.

The majority of this research was carried out during Georges Habib's 3-month visit to Durham University in 2024, which was supported by the Atiyah Lebanon-UK Fellowship AF-2023-01 of the LMS (London Mathematical Society), a Visiting Fellowship from the ICMS (International Centre for Mathematical Sciences) and CAMS (Center for Advanced Mathematical Sciences) at the American University of Beirut.
Part of this project was carried out when the authors met in Bristol. The authors acknowledge the EPSRC grant EP/T030577/1 which made this visit possible.\\

\end{document}